\documentclass[11pt, a4paper]{article} 

\usepackage[utf8]{inputenc}
\usepackage{graphicx}
\usepackage{import}
\usepackage{hyperref}
\usepackage{algorithm}
\usepackage{algorithmic}
\usepackage{multirow}
\usepackage{tikz}
\usepackage{amssymb}
\usepackage{mathtools}
\usepackage{amsthm}
\usepackage{pdfpages}
\usepackage{subcaption}
\usepackage[labelsep=endash]{caption}
\usepackage{pgfplots}
\usepackage{bm}
\usepackage{subcaption}
\usepackage{multicol}
\usepackage{bbm}
\usepackage{setspace}
\usepackage{enumitem}

\newcommand*\EE[1]{\ensuremath{{\mathbb{#1}}}}
\newcommand*\bb[1]{\ensuremath{\mathbf{#1}}}
\newcommand*\dd{\ensuremath{{\mathrm{d}}}}

\newtheorem{theorem}{Theorem}
\newtheorem{proposition}{Proposition}
\newtheorem{corollary}{Corollary}
\newtheorem{lemma}{Lemma}
{\theoremstyle{definition}
\newtheorem{assumption}{Assumption}
\newtheorem{remark}{Remark}
}

\newcommand{\N}[0]{\mathbb{N}}

\newcommand{\R}[0]{\mathbb{R}}
\newcommand{\C}[0]{\mathbb{C}}

\newcommand{\bbE}{\mathbb{E}}

\newcommand{\bsa}{\mathbf{a}}

\newcommand{\bsc}{\mathbf{c}}

\newcommand{\bsx}{\mathbf{x}}

\newcommand{\bsz}{\mathbf{z}}

\newcommand{\bsZ}{\mathbf{Z}}

\newcommand{\bsomega}{\boldsymbol{\omega}}
\newcommand{\bstau}{\boldsymbol{\tau}}
\newcommand{\bsDelta}{\boldsymbol{\Delta}}

\newcommand{\calA}{\mathcal{A}}
\newcommand{\calB}{\mathcal{B}}

\newcommand{\calE}{\mathcal{E}}
\newcommand{\calF}{\mathcal{F}}

\newcommand{\amin}{a_\mathrm{min}}
\newcommand{\amax}{a_\mathrm{max}}
\newcommand{\kapmin}{\kappa_\mathrm{min}}
\newcommand{\kapmax}{\kappa_\mathrm{max}}
\newcommand{\lambdahat}{\widehat{\lambda}}
\newcommand{\rd}{\mathrm{d}}
\newcommand{\dist}{\mathrm{dist}}
\newcommand{\length}{\mathrm{length}}

\renewcommand{\Im}{\mathrm{Im}}
\renewcommand{\Re}{\mathrm{Re}}

\DeclareFontFamily{U}{mathx}{\hyphenchar\font45}
\DeclareFontShape{U}{mathx}{m}{n}{
      <5> <6> <7> <8> <9> <10>
      <10.95> <12> <14.4> <17.28> <20.74> <24.88>
      mathx10
      }{}
\DeclareSymbolFont{mathx}{U}{mathx}{m}{n}
\DeclareFontSubstitution{U}{mathx}{m}{n}
\DeclareMathAccent{\widecheck}{0}{mathx}{"71}

\makeatletter
\let\@fnsymbol\@arabic
\makeatother

\title{Multilevel Monte Carlo methods for stochastic convection-diffusion eigenvalue problems}

\date{\today}

\author{Tiangang Cui\footnotemark[1] \and
        Hans De Sterck\footnotemark[2]\and
	    Alexander D.~Gilbert\footnotemark[3] \and
	    Stanislav Polishchuk\footnotemark[4] \and
        Robert Scheichl\footnotemark[5]
}

\begin{document}

\footnotetext[1]{School of Mathematics and Statistics, The University of Sydney, NSW 2006, Australia. \\\texttt{tiangang.cui@sydney.edu.au}
}

\footnotetext[2]{Department of Applied Mathematics, University of Waterloo,
Ontario, Canada N2L 3G1.\\\texttt{hdesterck@uwaterloo.ca}
}

\footnotetext[3]{School of Mathematics and Statistics, University of New South Wales, Sydney, NSW 2052, Australia. \\ \texttt{alexander.gilbert@unsw.edu.au}
}

\footnotetext[4]{School of Mathematics, Monash University, Victoria 3800, Australia. \\\texttt{stanislav.polishchuk@monash.edu}
}
\footnotetext[4]{
Institute of Applied Mathematics and Interdisciplinary Center for Scientific Computing (IWR), Universität Heidelberg, Im Neuenheimer Feld 205,
69120 Heidelberg, Germany.
\\
\texttt{r.scheichl@uni-heidelberg.de}
}

\maketitle

\begin{abstract}
We develop new multilevel Monte Carlo (MLMC) methods to estimate the expectation of the smallest eigenvalue of a stochastic convection-diffusion operator with random coefficients. The MLMC method is based on
a sequence of finite element (FE) discretizations of the eigenvalue problem on a hierarchy of 
increasingly finer meshes.
For the discretized, algebraic eigenproblems we use both
the Rayleigh quotient (RQ) iteration and implicitly restarted Arnoldi (IRA), providing an analysis of the cost in each case.
By studying the variance on each level and adapting classical FE error bounds to the stochastic setting, 
we are able to bound the total error
of our MLMC estimator and provide
a complexity analysis. 
As expected, the complexity bound for our MLMC estimator is superior to plain Monte Carlo.
To improve the efficiency of the MLMC further, we exploit the hierarchy of meshes and use coarser approximations as starting values for the eigensolvers on finer ones.
To improve the stability of the MLMC
method for convection-dominated problems, we employ two additional strategies.
First, we consider the
streamline upwind Petrov--Galerkin formulation of the discrete eigenvalue problem, which allows us to start the MLMC method on coarser meshes than is possible with standard FEs.
Second, we apply a homotopy method to add stability to the eigensolver for each sample.
Finally, we present a multilevel quasi-Monte Carlo method that replaces 
Monte Carlo with a quasi-Monte Carlo (QMC) rule on each level.
Due to the faster convergence of QMC, this improves the overall complexity.
We provide detailed numerical results
comparing our different strategies
to demonstrate the practical feasibility
of the MLMC method in different use cases.
The results support our complexity analysis and further demonstrate the superiority over plain Monte Carlo 
in all cases.


\end{abstract}

\section{Introduction}
\label{sec:int}
We consider the following convection-diffusion eigenvalue problem with random coefficients:
Find a non-trivial eigenpair $(\lambda,u)\in\mathbb{C}\times H^1_0(D;\mathbb{C})$
such that
\begin{align}
\label{convdiffeq}
	-\nabla\cdot\big(\kappa(\mathbf{x},\bsomega)\nabla u(\mathbf{x},\bsomega)\big)
	+ \mathbf{a}(\mathbf{x},\bsomega)\cdot\nabla u(\mathbf{x},\bsomega)
	 &= \lambda(\bsomega) u(\mathbf{x},\bsomega).
\end{align}
The PDE is considered for the \emph{physical variable} $\bsx$ in a bounded Lipschitz domain $D\in\mathbb{R}^d$ with $d=1,2,$ or $3$, 
and for the \emph{stochastic variable} $\bsomega$ from a given probability space $(\Omega,\calF,\pi)$.
For $\pi$-almost all $\bsomega \in \Omega$, 
we assume Dirichlet conditions on 
the boundary $\Gamma = \partial D$,
\begin{equation*}
	u(\bsx,\bsomega) = 0
 \quad \text{for } \bsx \in \Gamma.
\end{equation*}

The conductivity $\kappa(\mathbf{x},\bsomega):D\times\Omega\rightarrow\mathbb{R}$ is a log-uniform random field 
(as used in, e.g., \cite{DGLS19}), defined using the process convolution approach in~\cite{Higdon2002}, such that 
\begin{equation}
\label{eq:log-kap}
\log\kappa(\mathbf{x},\bsomega)=
Z(\bsx, \bsomega) = 
\sum_i\omega_ik(\mathbf{x}-\bsc_i),
\end{equation}
with $k(\bsx-\bsc_i)$ a kernel centered at a certain number of points $\bsc_i \in D$ and i.i.d. uniform random variables $\omega_i\sim\mathcal{U}[0,1]$. Similarly, the convection velocity $\mathbf{a}(\mathbf{x},\bsomega):D\times\Omega\rightarrow\mathbb{R}^d$ can also be some bounded random variable,
which also depends on uniform
random variables $\omega_i \sim \mathcal{U}[0, 1]$ and is additionally assumed to be divergence-free, i.e.,
\begin{equation}
\nabla\cdot\mathbf{a}(\mathbf{x},\bsomega) = 0.
\end{equation}

The purpose of this paper is to compute the expectation of the smallest eigenvalue of~\eqref{convdiffeq}, 
\begin{equation}
\label{eq:E[lambda]}
\mathbb{E}[\lambda]=\int_\Omega\lambda(\bsomega)\,\mathrm{d}\pi(\bsomega),
\end{equation}
using multilevel Monte Carlo methods.

Stochastic eigenvalue problems arise in a variety of physical and scientific applications and their numerical simulations.
Factors such as measurement noise, limitations of mathematical models, the existence of hidden variables, the randomness of input parameters, and other factors contribute to uncertainties in the modelling and prediction of many phenomena.
Applications of uncertainty quantification (UQ) specifically related to eigenvalue problems include:
nuclear reactor criticality calculations~\cite{Avramova2010,Ayres2012,nuclear_1976},
the derivation of the natural frequencies of an aircraft or a naval vessel~\cite{Kana2016}, band gap calculations in photonic crystals~\cite{Dobson2000,Giani2012,Scheichl2013},
the computation of ultrasonic resonance frequencies to detect the presence of gas hydrates~\cite{rus_hydrate}, the analysis of the elastic properties of crystals with the use of rapid measurements~\cite{migliori2016resonant,rus_crystals}, or the calculation of acoustic vibrations~\cite{Carnoy1983,Thomson1981}.
Stochastic convection-diffusion equations are used to describe simple cases of turbulent~\cite{drummond1984,kraich1970,morton1996,stynes2005} or subsurface flows~\cite{Tart2011,Zhang2002}.

Monte Carlo sampling is one of the most popular methods for quantifying uncertainties in quantities of interest coming from stochastic PDEs. 
Although simple and robust, Monte Carlo methods can be severely inefficient when applied to UQ problems, because their slow convergence rate often requires
a large number of samples to meet the desired accuracy.
To improve the efficiency, the multilevel Monte Carlo (MLMC) method was developed, where the
key idea is to reduce the computational cost by spreading the samples over a hierarchy of
discretizations.
The main idea was introduced by Heinrich in 2001~\cite{Heinrich2001} for path integration, then generalized by Giles in 2008~\cite{Giles2008}
for SDEs.
More recently, MLMC methods have been applied with great 
success to stochastic PDEs, see, e.g., \cite{Barth2011,Beck2018,Cliffe2011,Mishra2012,RobMLMC2017,Teckentrup2012} and 
\cite{Gilbert2020,Gilbert2021} specifically for eigenproblems.
A general overview of MLMC is presented by Giles in \cite{maingiles}.

In this paper, we present a MLMC method to approximate \eqref{eq:E[lambda]}, 
which, motivated by the use of MLMC for source problems described above, is based on a hierarchy of discretizations of the eigenvalue problem \eqref{convdiffeq} and which is much more efficient in practice than a Monte carlo approximation.
We consider two discretization methods, a standard Galerkin finite element method (FEM) and 
a streamline upwind Petrov--Galerkin (SUPG) method. The SUPG method improves the stability of the approximation
for cases with high convection and also allows us to start the MLMC method from a coarser discretization.
To further reduce the cost of our MLMC method, we again exploit the hierarchy of discretizations 
by using approximations on coarse levels as the starting values for the eigensolver on the fine level.
We also present the two extensions of MLMC that aim to improve different aspects of the method.
First, to improve the stability of the eigensolver for each sample we include a homotopy method for solving convection-diffusion eigenvalue problems in the MLMC algorithm.
The homotopy method computes the eigenvalue of the convection-diffusion operator by following a continuous path starting from the pure diffusion operator.
Second, to improve the overall complexity we present a  multilevel quasi-Monte Carlo method that aims to speed up the convergence of the variance on each level by replacing the Monte Carlo samples with a quasi-Monte Carlo (QMC) quadrature rule.

The structure of the paper is as follows.
Section \ref{sec:var} introduces 
the variational formulation of \eqref{convdiffeq}, along with necessary
background material on stochastic convection-diffusion eigenvalue problems.
Two discrete formulations of the eigenvalue problem are introduced: the Galerkin FEM and 
the SUPG method.
Section \ref{sec:mlmc} introduces the MLMC method and presents the corresponding complexity analysis.
In particular, this section details how to efficiently use 
each eigensolver, the Rayleigh quotient and implicitly restarted Arnoldi iterations, within the MLMC algorithm.
In Section~\ref{sec:ext_mlmc}, we present the two extensions
of our MLMC algorithm: a homotopy MLMC and a multilevel quasi-Monte Carlo method.
Section \ref{sec:nums} presents numerical results for finding the smallest eigenvalue of the convection-diffusion operator in a variety of settings. In particular, we present examples
for difficult cases with high convection.

To ease notation, for the remainder of the paper
we combine the random variables in 
the convection and diffusion coefficients into a single uniform random vector of dimension $s < \infty$,  denoted by $\bsomega = (\omega_i)_{i = 1}^s$
with $\omega_i \sim \mathcal{U}[0, 1]$.
In this case, $\pi$ is the product uniform measure on $\Omega \coloneqq [0, 1]^s$.

\section{Variational formulation}
\label{sec:var}

The eigenvalue problem~(\ref{convdiffeq}) needs to be discretized, because its solution is not analytically tractable for arbitrary geometries and parameters.
As such, we apply the standard finite element method to (\ref{convdiffeq}) to obtain an approximation of the desired eigenpair $(\lambda, u)$.

Before deriving the variational form of \eqref{convdiffeq}, 
we first establish certain assumptions about the problem domain, the random field $\kappa(\bsomega)$ and the velocity field $\mathbf{a}(\bsomega)$ for $\bsomega\in\Omega$,
which, in particular, ensure that the solution is in $H^2(D)$ \cite{Gris11} as well as incompressibility.

\begin{assumption}
  Assume that $D \subset \R^{d}$, for $d = 1, 2, $ or $3$, is a bounded, convex domain with Lipschitz continuous boundary $\Gamma$. 
	\label{ass:1}
\end{assumption}
\begin{assumption}
	The diffusion coefficient is bounded from above and from below for almost all $\bsomega \in \Omega$, i.e., there exist two constants $\kapmin , \kapmax$ 
	such that
	$0<\kappa_{\min}\leq\kappa(\mathbf{x},\bsomega) \leq\kappa_{\max}<\infty$.
 In addition, we assume that also
 $\|\kappa(\cdot, \bsomega)\|_{W^{1, \infty}} \leq \kapmax$ for almost all $\bsomega \in \Omega$.
	\label{ass:2}
\end{assumption}

\begin{assumption}
\label{ass:3}
The convection coefficient is divergence free, $\nabla \cdot \bsa(\bsx, \bsomega) = 0$ for all 
$\bsx \in D$, and uniformly bounded, $\|\bsa(\cdot, \bsomega)\|_{L^\infty} \leq \bsa_\mathrm{max}$,
for almost all $\bsomega$.
\end{assumption}

A simple example of a random convection term
is a homogeneous convection, 
$\bsa(\bsx, \bsomega) = [a_1\omega_1,\ldots, a_d \omega_d]^\top$ 
for $a_1, \ldots, a_d \in \R$, which are independent of $\bsx$.
Another example is the curl of random vector field, e.g., 
$\bsa(\bsx, \bsomega) = \nabla \times \bsZ(\bsx, \bsomega)$
where $\bsZ$ is a vector-valued random
field similar to that defined in \eqref{eq:log-kap}. 
Both of these examples satisfy Assumption~\ref{ass:3}.

Next we introduce the variational form of \eqref{convdiffeq}. Whenever it does not lead to confusion, we drop the spatial coordinate of (stochastic) functions for brevity---for example, $u(\mathbf{x},\bsomega)$ is also written as $u(\bsomega)$.
Let $V=H^1_0(\Omega)$ be the first-order Sobolev space of complex-valued functions with vanishing trace on the boundary with norm $\|v\|_V=\|\nabla v\|_{L^2}$.
Then let $V^*$ denote the dual space of $V$.
Multiplying \eqref{convdiffeq} by a test function $v\in V$
and then performing integration by parts, noting that we have no Neumann boundary condition term since $u(\bsx,\bsomega )=0$ on $\Gamma$, we obtain
\begin{align*}
\int_D\mathbf{a}(\mathbf{x},\bsomega)\cdot\nabla u(\mathbf{x},\bsomega)
v(\mathbf{x})\,\dd\mathbf{x}
+\int_D\kappa(\mathbf{x},\bsomega)\nabla u(\mathbf{x},\bsomega)\cdot\nabla v(\mathbf{x})\,\dd\mathbf{x} \quad  & \\
=\lambda(\bsomega)\int_D u(\mathbf{x},\bsomega)v(\mathbf{x})\,\dd\mathbf{x}.
\end{align*}
The variational eigenvalue problem corresponding to \eqref{convdiffeq} is
then:
Find a non-trivial eigenpair $(\lambda(\bsomega),u(\bsomega))\in\mathbb{C}\times V$ with 
$\| u(\bsomega)\|_{L^2} = 1$ such that
\begin{equation}
\label{eq:varevp}
	\mathcal{A}(\bsomega;  u(\bsomega),v)
	+\mathcal{B}(\bsomega; u(\bsomega),v)
	=\lambda(\bsomega) \langle u(\bsomega),v\rangle
	\mspace{15mu}\forall v\in V,
\end{equation}
where
\begin{align*}
	\mathcal{A}(\bsomega; u(\bsomega),v)
	&\coloneqq\int_D\kappa(\bsx,\bsomega)
	\nabla u(\bsx,\bsomega)\cdot\nabla \overline{v(\bsx)} \, \dd\mathbf{x},
\\
	\mathcal{B}(\bsomega; u(\bsomega),v)
	&\coloneqq\int_D\mathbf{a}(\bsx,\bsomega)
	\cdot\nabla u(\bsx,\bsomega) \overline{v(\bsx)}\,\dd\mathbf{x},
\end{align*}
and $\langle \cdot, \cdot \rangle$ denotes the $L^2(D)$ inner product
\[
	\langle u(\bsomega),v \rangle \coloneqq
	\int_D u(\bsx,\bsomega)\overline{v(\bsx)} \, \dd\mathbf{x}.
\]

Since the velocity $\mathbf{a}$ is divergence free, 
$\nabla \cdot \mathbf{a} = 0$,
the sesquilinear form in \eqref{eq:varevp} is uniformly coercive, i.e.,
\begin{equation}
\label{eq:coerc}
    \calA(\bsomega; v, v) + \calB(\bsomega; v, v) \geq \amin \|v\|_V^2, \quad \forall v \in V,
\end{equation}
with $\amin > 0$ independent of $\bsomega$. It is also uniformly bounded,
i.e.,
\begin{equation}
\label{eq:bounded}
    \calA(\bsomega; v, z) + \calB(\bsomega; v, z) \leq \amax \|v\|_V \|z\|_V, \quad  \forall v, z \in V,
\end{equation}
with $\amax < \infty$ independent of $\bsomega$.

For each $\bsomega \in \Omega$, the eigenvalue problem \eqref{eq:varevp} 
admits a countable sequence of eigenvalues $(\lambda_k(\bsomega))_{k = 1}^\infty \subset \C$,
which has no finite accumulation points,
and the smallest eigenvalue, $\lambda_1(\bsomega)$, is real and simple, see, e.g., \cite{BO91}. The eigenvalues are enumerated in order of increasing magnitude, 
counting multiplicity,
such that
\[
0 < \lambda_1(\bsomega) < |\lambda_2(\bsomega)| \leq |\lambda_3(\bsomega)|\leq \cdots
\]
with corresponding eigenfunctions $(u_k(\cdot,\bsomega))_{k = 1}^\infty$, enumerated accordingly.

In addition to the primal form \eqref{eq:varevp}, 
to facilitate our analysis later on we also
consider the dual eigenproblem:
Find a non-trivial dual eigenpair $(\lambda^*(\bsomega),u^*(\bsomega))\in\mathbb{C}\times V$ with $\|u^*(\bsomega)\|_{L^2}=1$ such that
\begin{equation}
\label{eq:dualevp}
	\mathcal{A}(\bsomega; v, u^*(\bsomega)) + \mathcal{B}(\bsomega; v, u^*(\bsomega)) =\overline{\lambda^*(\bsomega)} \langle v, u^*(\bsomega)\rangle
	\mspace{15mu}\forall v\in V.
\end{equation}
The primal and dual eigenvalues are related to each other via $\lambda(\bsomega)=\overline{\lambda^*(\bsomega)}$.

\begin{proposition}
For all $\bsomega \in \Omega$, the smallest eigenvalue $\lambda_1(\bsomega)$
of \eqref{eq:varevp} is simple and the gap is uniformly bounded, i.e.,
there exists $\rho > 0$, independent of $\bsomega$, such that
\begin{equation}
    |\lambda_2(\bsomega) -\lambda_1(\bsomega)| \geq \rho.
\end{equation}
\end{proposition}
\begin{proof}
For each $\bsomega \in \Omega$, the Krein--Rutman Theorem implies that 
$\lambda_1(\bsomega)$ is simple. It remains to show that
the gap is uniformly bounded for $\bsomega \in \Omega$.
Since the eigenvalues are continuous in $\bsomega$,
it follows that the gap is also continuous.
Hence, there exists a strictly positive minimum on the compact domain $\Omega$
and we can take
\[
\rho \coloneqq \min_{\bsomega \in \Omega} 
|\lambda_2(\bsomega) - \lambda_1(\bsomega)| > 0.
\]
\end{proof}

\begin{theorem}\label{thm:u-reg}
Suppose Assumptions \ref{ass:1}--\ref{ass:3} hold.
For $\bsomega \in \Omega$, let $(\lambda(\bsomega), u(\cdot, \bsomega))$
be an eigenpair of the EVP \eqref{eq:varevp} and let 
$(\lambda^*(\bsomega), u^*(\cdot, \bsomega))$ be the corresponding
dual eigenpair of the adjoint EVP \eqref{eq:dualevp}, 
i.e., $\lambda(\bsomega) = \overline{\lambda^*(\bsomega)}$.
Then, the primal and the dual eigenfunctions satisfy
$u(\cdot, \bsomega),\ u^*(\cdot, \bsomega) \in V \cap H^2(D)$ with
\begin{equation}
\label{eq:u_H2}
    \|u(\bsomega)\|_{H^2} \,\leq\, 
    C_{\lambda, 2} |\lambda(\bsomega)|
    \quad \text{and} \quad
    \|u^*(\bsomega)\|_{H^2} \,\leq\, 
    C_{\lambda^*, 2} |\lambda^*(\bsomega)|,
\end{equation}
for $C_{\lambda, 2} < \infty$  and $C_{\lambda^*, 2} < \infty$
independent of $\bsomega$.
\end{theorem}

\begin{proof}
Rearranging \eqref{convdiffeq}, we can write the Laplacian of $u(\cdot, \bsomega)$
as
\begin{align*}
- \Delta u(\bsx, \bsomega) & = \frac{1}{\kappa(\bsx, \bsomega)} \big( \nabla \kappa(\bsx, \bsomega) \cdot \nabla u(\bsx, \bsomega)
- \bsa(\bsx, \bsomega) \cdot \nabla u(\bsx, \bsomega) + \lambda(\bsomega) u(\bsx, \bsomega)\big)\\ 
& \eqqcolon f_{\bsomega}(\bsx),
\end{align*}
which holds for almost all $\bsx \in D$.
Since $\kappa(\cdot, \bsomega) \in W^{1, \infty}(D)$, $\bsa(\cdot, \bsomega) \in L^\infty(D)^d$,
$u(\cdot, \bsomega) \in V$ and $1/\kappa(\bsx, \bsomega) \leq 1/\kapmin < \infty$ it follows that 
$f_{\bsomega} \in L^2(D)$ with
\begin{align*}
\|f_{\bsomega}\|_{L^2} \,&\leq\, \frac{1}{\kapmin}
\big( \|\kappa(\bsomega)\|_{W^{1, \infty}} \|u(\bsomega)\|_V
+ \|\bsa(\bsomega) \|_{L^\infty} \|u(\bsomega)\|_V + |\lambda(\bsomega)|\big)
\\
&\leq\, \frac{1}{\kapmin}\Big(
\big( \kapmax + \bsa_\mathrm{max}\big) \, \|u(\bsomega)\|_V + |\lambda(\bsomega)| \Big),
\end{align*}
where in the last step we have used that $\|u(\bsomega)\|_{L^2} = 1$, as well as Assumptions~\ref{ass:2} and \ref{ass:3}. Since $\lambda(\bsomega), u(\cdot, \bsomega)$ satisfy \eqref{eq:varevp} with 
$\|u(\bsomega)\|_{L^2} = 1$ and the sesquilinear form is coercive, it follows from \eqref{eq:coerc} that
\begin{align*}
& |\lambda(\bsomega)| = |\calA(\bsomega; u(\bsomega), u(\bsomega)) + \calB(\bsomega; u(\bsomega), u(\bsomega))|
\,\geq\, \amin \|u(\bsomega)\|_V^2  \,\geq\, \amin C_{\mathrm{Poin}}^2, 
\end{align*}
where in the last inequality we have used Poincar\'e's inequality, as well as 
$\|u(\bsomega)\|_{L^2} = 1$ again. 
The first inequality also implies $\|u(\bsomega)\|_V \leq \sqrt{|\lambda(\bsomega)|/\amin}$. Thus, substituting these two bounds, the $L^2$-norm of $f_{\bsomega}$ is bounded by
\begin{equation}
\label{eq:f_omega_L2}
\|f_{\bsomega}\|_{L^2} \,\leq\, \frac{1}{\kapmin}
\bigg( \frac{\kapmax + \bsa_\mathrm{max}}{\amin C_\mathrm{Poin}} + 1 \bigg) |\lambda(\bsomega)|,
\end{equation}
where the constant is independent of $\lambda$.

Finally, using classical results in Grisvard \cite{Gris11}
it follows that
\[
\|u(\bsomega) \|_{H^2} \,\leq\, C_D \|\Delta u(\bsomega)\|_{L^2}
\,=\, C_D \|f_{\bsomega} \|_{L^2},
\]
where $C_D$ depends only on the domain $D$. Finally, substituting
in the bound on $\|f_{\bsomega}\|_{L^2}$ \eqref{eq:f_omega_L2} gives the 
desired upper bound \eqref{eq:u_H2}.

The result for the dual eigenfunction follows analogously.
\end{proof}

\subsection{Finite element formulation}
\label{sec:fem}

Let $\{\mathcal{T}_h\}_{h>0}$ be a family of (quasi-)uniform, shape-regular, conforming meshes on the spatial domain $D$, where each $\mathcal{T}_h$ is parameterised by its mesh width $h>0$.
For $h>0$, we approximate the infinite-dimensional space $V$ by a 
finite-dimensional subspace $V_h$.
In this paper, we consider piecewise linear finite element (FE) spaces, but 
the method will work also for more general spaces.

The resulting discrete variational problem is to
find non-trivial primal and dual eigenpairs $(\lambda(\bsomega),u_h(\bsomega))\in\mathbb{C}\times V_h$ and  $(\lambda^*(\bsomega),u_h^*(\bsomega))\in\mathbb{C}\times V_h$ such that
\begin{equation}
\label{eq:evp-fe}
	\mathcal{A}(\bsomega; u_h(\bsomega),v_h)+\mathcal{B}(\bsomega; u_h(\bsomega),v_h)=\lambda_h(\bsomega) \langle u_h(\bsomega),v_h \rangle,
\quad \forall v_h \in V_h\,,
\end{equation}
and
\begin{equation}
	\mathcal{A}(\bsomega; v_h,u_h^*(\bsomega))+\mathcal{B}(\bsomega; v_h,u_h^*(\bsomega))= \overline{\lambda^*_h}(\bsomega) \langle v_h,u_h^*(\bsomega) \rangle, \quad \forall v_h \in V_h\,.
\end{equation}
For each $\bsomega$, it is well-known that for $h$ sufficiently small 
the FE eigenvalue problem \eqref{eq:evp-fe}
admits $M_h \coloneqq \dim(V_h)$ eigenpairs, 
denoted by
\begin{equation}
\label{eq:fe_evals}
\big(\lambda_{h, 1}(\bsomega),u_{h, 1}(\bsomega)\big), \, 
\big(\lambda_{h, 2}(\bsomega),u_{h, 2}(\bsomega)\big), \ldots, \big(\lambda_{h, M_h}(\bsomega),u_{h, M_h}(\bsomega) \big) \, \in \, \C \times V_h\,, 
\end{equation}
which approximate the first $M_h$ eigenpairs of \eqref{eq:varevp}. This approach is also called the Galerkin method.

\label{sec:fem_error}
In convection-dominated regions, the Galerkin method 
has well-known stability issues for standard (Lagrange-type) FEs, if the element size $h$ does not capture all necessary information about the flow.
The Peclet number (sometimes called the mesh Peclet number)~\cite{Zie2000b}
\begin{equation}\label{eq:Peclet}
	\mathrm{Pe}(\mathbf{x},\bsomega)=\frac{|\mathbf{a}(\mathbf{x},\bsomega)|h}{2\kappa(\mathbf{x},\bsomega)}
\end{equation}
governs how small the mesh size $h$ should be in order to have a stable solution using basic (Lagrange-type) FE methods.

The error in the FE approximations \eqref{eq:fe_evals} can be analysed using the 
Babu\v{s}ka--Osborn theory \cite{BO91}. We state the error
bounds for a simple eigenpair.

\begin{theorem}
\label{thm:fe}
Let $(\lambda(\bsomega), u(\bsomega))$ be an eigenpair of \eqref{eq:varevp} 
that is simple for all $\bsomega \in \Omega$, where $\Omega$ is a compact domain. 
Then there exist constants $C_\lambda, C_u$, independent of $h$ and $\bsomega$,
such that
\begin{equation}
\label{eq:lam-fe-err}
|\lambda(\bsomega) - \lambda_h(\bsomega)| \,\leq\, C_\lambda h^2
\end{equation}
and $u_h(\bsomega)$ can be normalized such that
\begin{equation}
\label{eq:u-fe-err}
\|u(\bsomega) - u_h(\bsomega)\|_V \,\leq\, C_u h.
\end{equation}
\end{theorem}
\begin{proof}
See Babu\v{s}ka and Osborn \cite{BO91} and the appendix, where we show explicitly that the constants
are bounded uniformly in $\bsomega$.
\end{proof}

\subsection{Streamline-upwind Petrov--Galerkin formulation}
\label{sec:supg}
A sufficiently small Peclet number \eqref{eq:Peclet} guarantees numerical stability of the standard Galerkin method.
One can either choose a small overall mesh size $h$ or locally adapt the mesh size to satisfy the stability condition. However, globally reducing the mesh size may lead to a high computational cost, while local adaptations may need to be performed path-wise for each realisation of $\bsomega$, which in turn leads to complications in the algorithmic design. In this section, we consider using the streamline-upwind Petrov--Galerkin (SUPG) method to improve numerical stability.

The SUPG method was introduced by Brooks and Hughes~\cite{1982SUPG} to stabilize the finite element solution.
Since then, the method has been extensively investigated  and used in various applications~\cite{Bochev2004,Cohen2012,Hauke2002,Hughes1984,Hughes1986,Knobloch2008}.
The SUPG method can be derived in several ways.
Here, we introduce its formulation by adding a stabilization term to the bilinear form.
An equivalent weak formulation can be obtained by defining a test space with additional test functions in the form $\hat{v}(\mathbf{x})=v(\mathbf{x})+p(\mathbf{x})$,
where $v(\mathbf{x})$ is a standard test function in the finite element method and $p(\mathbf{x})$ is an additional discontinuous function.

We define the residual operator $\mathcal{R}$ as
\begin{equation}
\mathcal{R}(\bsomega, \sigma) v = \mathbf{a}(\bsomega)\cdot\nabla v -\nabla\cdot\kappa(\bsomega)\nabla v -\sigma v,\label{eq:resid}
\end{equation}
which gives the residual of the convection-diffusion equation \eqref{convdiffeq} for a pair $(\sigma, v) \in \mathbb{C}\times V$. Then, stabilization techniques can be derived from the general formulation
\begin{equation}
\begin{aligned}
& \mathcal{A}(\bsomega; u(\bsomega), v)+\mathcal{B}(\bsomega; u(\bsomega),v) \\
& \qquad +\sum_{m=1}^{|\mathcal{T}_h|}\int_{D_m}\tau_m(\mathbf{x},\bsomega)\left(\mathcal{R}(\bsomega, \lambda(\bsomega)) u(\bsx, \bsomega) \right)
\left(\mathcal{P}(\bsomega) v(\bsx) \right)
\dd\mathbf{x} \\ 
& \qquad =\lambda(\bsomega)\langle u(\bsomega),v\rangle,
\end{aligned}
\label{eq:stab}
\end{equation}
where $|\mathcal{T}_h|$ is the number of elements of the mesh 
$\mathcal{T}_h$, $\mathcal{P}(\bsomega)$ is some stabilization operator and
$\tau_m(\bsomega)$ is the stabilization parameter acting in the $m$th finite 
element. The stabilization strategy will be determined by $\mathcal{P}(\bsomega)$ and $\tau_m(\bsomega)$.

Various definitions exist for the operator $\mathcal{P}(v,\bsomega)$, such as the Galerkin Least Square method~\cite{Hughes1989}, the SUPG method~\cite{Broesen2014,1982SUPG,donea2003}, the Unusual Stabilized Finite Element method~\cite{Usfem2002}, etc. For the SUPG method, the stablization operator $\mathcal{P}(\bsomega)$ is defined as
\begin{equation}
\mathcal{P}(\bsomega) v =\mathbf{a}(\bsomega)\cdot\nabla v.
\label{eq:supg}
\end{equation}
Substituting Equations (\ref{eq:resid}) and (\ref{eq:supg}) into (\ref{eq:stab}) gives the SUPG weighted residual formulation
\begin{align*}
& \mathcal{A}(\bsomega; u(\bsomega), v)+\mathcal{B}(\bsomega; u(\bsomega),v) + \sum_{m=1}^{|\mathcal{T}_h|}\int_{D_m}\Big( \tau_m(\mathbf{x},\bsomega) \big(\mathbf{a}(\mathbf{x},\bsomega)\cdot\nabla u(\mathbf{x},\bsomega) \\
& \qquad - \nabla\cdot\kappa(\mathbf{x},\bsomega)\nabla u(\mathbf{x},\bsomega)-\lambda(\bsomega) u(\mathbf{x},\bsomega)\big)(\mathbf{a}(\mathbf{x},\bsomega)\cdot\nabla v(\mathbf{x})) \Big)\dd \mathbf{x}\\
& \qquad =\lambda(\bsomega)\langle u(\bsomega),v\rangle,
\end{align*}
which is equivalent to 
\begin{equation}
\label{eq:weak_supg}
\begin{aligned}
& \mathcal{A}(\bsomega; u(\bsomega), v)+\mathcal{B}(\bsomega; u(\bsomega),v) + \sum_{m=1}^{|\mathcal{T}_h|}\int_{D_m}\Big( \tau_m(\mathbf{x},\bsomega) \big(\mathbf{a}(\mathbf{x},\bsomega)\cdot\nabla u(\mathbf{x},\bsomega) \\
& \qquad - \nabla\cdot\kappa(\mathbf{x},\bsomega)\nabla u(\mathbf{x},\bsomega)\big)(\mathbf{a}(\mathbf{x},\bsomega)\cdot\nabla v(\mathbf{x})) \Big)\dd \mathbf{x} \\
&\qquad =\lambda(\bsomega)\bigg(\langle u(\bsomega),v\rangle +\sum_{m=1}^{|\mathcal{T}_h|}\int_{D_m}\tau_m(\mathbf{x},\bsomega) u(\mathbf{x},\bsomega) \mathbf{a}(\mathbf{x},\bsomega)\cdot\nabla v(\mathbf{x})\dd \mathbf{x} \bigg).
\end{aligned}
\end{equation}

After approximating the weak form~\eqref{eq:weak_supg} by the usual finite-dimensional subspaces, we obtain the discrete variational problem: Find non-trivial (primal)
eigenpairs $(\lambda_h(\bsomega),u_h(\bsomega))\in\mathbb{C}\times V_h$ such that
\begin{align}
& \mathcal{A}(\bsomega; u_h(\bsomega),v_h) +\mathcal{B}(\bsomega; u_h(\bsomega),v_h) + \sum_{m=1}^{|\mathcal{T}_h|}\int_{D_m} \Big( \tau_m(\mathbf{x},\bsomega) \big(\mathbf{a}(\mathbf{x},\bsomega)\cdot\nabla u_h(\mathbf{x},\bsomega) \nonumber\\
& \quad - \nabla\cdot\kappa(\mathbf{x},\bsomega)\nabla u_h(\mathbf{x},\bsomega)\big)(\mathbf{a}(\mathbf{x},\bsomega)\cdot\nabla v_h(\mathbf{x})) \Big) \dd \mathbf{x} \label{eq:evp-supg}\\
&\quad=\lambda_h(\bsomega)\bigg(\mathcal{M}(u_h(\bsomega),v_h) +\sum_{m=1}^{|\mathcal{T}_h|} \int_{D_m}\tau_m(\mathbf{x},\bsomega) u_h(\mathbf{x},\bsomega) \mathbf{a}(\mathbf{x},\bsomega)\cdot\nabla v_h(\mathbf{x})\dd \mathbf{x} \bigg), \nonumber
\end{align}
and dual eigenpairs $(\lambda_h^*(\bsomega),u_h^*(\bsomega))\in\mathbb{C}\times V_h$ such that
\begin{align}
& \mathcal{A}(\bsomega; v_h, u_h^*(\bsomega))+\mathcal{B}(\bsomega; v_h, u_h^*(\bsomega))+\sum_{m=1}^{|\mathcal{T}_h|}\int_{D_m}\Big(\tau_m(\mathbf{x},\bsomega) \big(\mathbf{a}(\mathbf{x},\bsomega)\cdot\nabla v_h(\mathbf{x}) \nonumber \\
& \quad -\nabla\cdot\kappa(\mathbf{x},\bsomega)\nabla v_h(\mathbf{x})\big)(\mathbf{a}(\mathbf{x},\bsomega)\cdot\nabla u_h^*(\mathbf{x},\bsomega))\Big)\dd \mathbf{x} \label{eq:evp-supg-dual} \\
&\quad = \overline{\lambda_h^*}(\bsomega)\Big(\mathcal{M}(v_h, u_h^*(\bsomega)) +\sum_{m=1}^{|\mathcal{T}_h|}\int_{D_m}\tau_m(\mathbf{x},\bsomega) v_h(\mathbf{x}) \mathbf{a}(\mathbf{x},\bsomega)\cdot\nabla u_h^*(\mathbf{x},\bsomega)\dd \mathbf{x} \Big).\nonumber
\end{align}
It follows that the right-hand side matrix is no longer symmetric and is stochastic compared to the mass matrix in the standard Galerkin method.

In general, finding the optimal stabilization parameter $\tau_m(\mathbf{x},\bsomega)$ is an open problem, and thus it is defined heuristically~\cite{Knobloch2008}.
We employ the following stabilization parameter~\cite{Bochev2004,Hauke2002}
\begin{equation}
\tau_m(\mathbf{x},\bsomega) = \frac{h_m}{2|\mathbf{a}(\mathbf{x},\bsomega)|}(\coth \mathrm{Pe}(\mathbf{x},\bsomega)-\frac{1}{\mathrm{Pe}(\mathbf{x},\bsomega)}).
\end{equation}
However, in practical implementations the following asymptotic expressions of $\tau_m(\mathbf{x}, \bsomega)$ are used
\begin{equation}
\hat{\tau}_m(\bsomega)=\left\{
  \begin{array}{ll}
  \displaystyle\max_{\mathbf{x}\in D_m} \frac{h_m}{2|\mathbf{a}(\mathbf{x},\bsomega)|},&
  \text{ if } \displaystyle \max_{\mathbf{x}\in D_m}\mathrm{Pe}(\mathbf{x},\bsomega)\geq 1, \\[6mm]  
  \displaystyle\max_{\mathbf{x}\in D_m} \frac{h_m^2}{12\kappa(\mathbf{x},\bsomega)},&
  \text{ if } \displaystyle\max_{\mathbf{x}\in D_m}\mathrm{Pe}(\mathbf{x},\bsomega)<1.
  \end{array}
\right.
	\label{eq:st_par}
\end{equation}

Figure~\ref{fig:spectr_supg} shows the 20 smallest eigenvalues for a single realization of random field $\kappa(\mathbf{x}, \bsomega)$ with velocity $\mathbf{a}(\mathbf{x},\bsomega)=[50,0]^T$ on meshes with size $h=2^{-3}, 2^{-4}, 2^{-5}$.
The standard Galerkin method has non-physical oscillations in the discretized eigenfunction for such a coarse mesh and its two smallest eigenvalues form a complex conjugate pair; this contradicts the fact that the smallest eigenvalue should be real and simple. The SUPG method, on the other hand, has a real smallest eigenvalue, indicating a stable solution.
\begin{figure}[h!]
		\centering
		\begin{subfigure}{0.45\textwidth}
		\centering
		\begin{tikzpicture}[scale=0.65]
\begin{axis}[
   axis lines = left,
   grid = both,
   xlabel = $\Re(\lambda_h)$,
	ylabel = {$\Im(\lambda_h)$},
   minor tick num=0,
   xmin = 0,
   xmax = 3000,
   ymin = -250,
   ymax = 250,
   ylabel near ticks,
   every axis y label/.style=
{at={(-0.2,0.5)},rotate=90},
   y tick label style={
       /pgf/number format/.cd,
           fixed,
       /tikz/.cd,
   },
   legend style={at={(0.85,0.9)}, anchor=north,legend columns=1},
]
	\addlegendentry{$h=2^{-3}$}
\addplot plot[only marks,mark=*,red,scatter src=explicit symbolic,domain=1:1e-30]
        file {results/spectr/spectr_supg_h3.txt};
        \addlegendentry{$h=2^{-4}$}
\addplot plot[only marks,mark=*,blue,scatter src=explicit symbolic,domain=1:1e-30]
        file {results/spectr/spectr_supg_h4.txt};
        \addlegendentry{$h=2^{-5}$}
\addplot plot[only marks,mark=x,red,scatter src=explicit symbolic,domain=1:1e-30]
        file {results/spectr/spectr_fem_h5.txt};

\end{axis}
\end{tikzpicture}
\caption{SUPG eigenvalues.}
\label{fig:spectr_supg2}
\end{subfigure}
\hfill
\begin{subfigure}{0.5\textwidth}
\centering
\begin{tikzpicture}[scale=0.65]
\begin{axis}[
   axis lines = left,
   grid = both,
   xlabel = $\Re(\lambda_h)$,
	ylabel = {$\Im(\lambda_h)$},
   minor tick num=0,
   xmin = 0,
   xmax = 3000,
   ymin = -250,
   ymax = 250,
   ylabel near ticks,
   every axis y label/.style=
{at={(-0.2,0.5)},rotate=90},
   y tick label style={
       /pgf/number format/.cd,
           fixed,
       /tikz/.cd,
   },
   legend style={at={(0.85,0.9)}, anchor=north,legend columns=1},
]
\addlegendentry{$h=2^{-3}$}
\addplot plot[only marks,mark=*,red,scatter src=explicit symbolic,domain=1:1e-30]
        file {results/spectr/spectr_fem_h3.txt};
        \addlegendentry{$h=2^{-4}$}
\addplot plot[only marks,mark=*,blue,scatter src=explicit symbolic,domain=1:1e-30]
        file {results/spectr/spectr_fem_h4.txt};
        \addlegendentry{$h=2^{-5}$}
\addplot plot[only marks,mark=x,red,scatter src=explicit symbolic,domain=1:1e-30]
        file {results/spectr/spectr_fem_h5.txt};

\end{axis}
\end{tikzpicture}
\caption{Finite element eigenvalues.}
\label{fig:spectr_supg6}
\end{subfigure}
\caption{The first 20 computed eigenvalues of the
SUPG (\emph{left}) and FEM (\emph{right}) discretizations of the convection-diffusion problem for $\kappa(\mathbf{x})=1$ and $\mathbf{a}=[50,0]^T$ using mesh sizes $h=2^{-3}, 2^{-4}, 2^{-5}$. }
\label{fig:spectr_supg}
\end{figure}
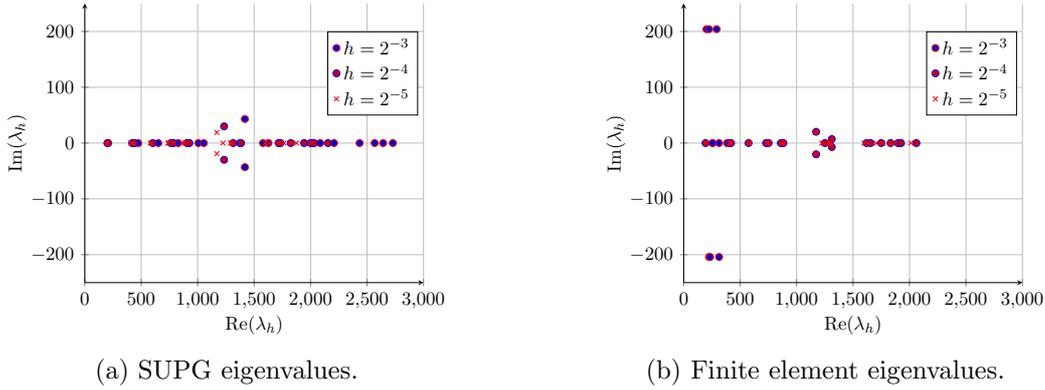

\section{Multilevel Monte Carlo methods}
\label{sec:mlmc}

To compute $\mathbb{E}[\lambda]$, we first approximate the eigenproblem \eqref{eq:varevp} for each $\bsomega \in \Omega$ 
and then use a sampling method to estimate the expected value of the approximate eigenvalue. There are two layers of approximation: First the  eigenvalue problem is discretized by a numerical method, e.g., FEM or SUPG as in Section~\ref{sec:fem}, then the resulting discrete 
eigenproblem is solved by an iterative eigenvalue solver, e.g., the Rayleigh quotient method, such that $\lambda(\bsomega) \approx \lambda_h(\bsomega) \approx \lambda_{h,K}(\bsomega)$, where $h$ denotes the meshwidth of the spatial discretization and $K$ denotes the number of iterations used by the eigenvalue solver. 

Applying the Monte Carlo method to $\lambda_{h, K}$, the expected eigenvalue can be approximated by the estimator
\begin{equation}\label{eq:mc}
\mathbb{E}[\lambda(\bsomega)] \approx Y_{h,K,N} \coloneqq \frac1N \sum_{n=1}^N \lambda_{h,K}(\bsomega_n),
\end{equation}
where the samples $\{\bsomega_n\}_{n = 1}^N \subset \Omega$ are i.i.d. uniformly on $\Omega$. This introduces a 
third factor that influences the accuracy of the estimator in \eqref{eq:mc} in addition to $h$ and $K$, namely the number of samples $N$. Note that we assume that the number of iterations $K$ is uniformly bounded in $\bsomega$.

The standard Monte Carlo estimator in \eqref{eq:mc} is computationally expensive. To measure its accuracy we use the mean squared error (MSE) 
\[
\mathrm{MSE}(\mathbb{E}[\lambda(\bsomega)], Y_{h,K,N}) = \mathbb{E}\left[\left|\mathbb{E}[\lambda(\bsomega)] - Y_{h,K,N}\right|^2\right],
\]
where the outer expectation is with respect to the samples in the estimator $Y_{h,K,N}$. Under mild conditions, the MSE can be decomposed as 
\[
\mathrm{MSE}(\mathbb{E}[\lambda], Y_{h,K,N}) = \left|\mathbb{E}[\lambda(\bsomega)] - \mathbb{E}[\lambda_{h,K}(\bsomega)]\right|^2 + \frac1N \mathrm{var}(\lambda_{h,K}(\bsomega)).
\]
In this decomposition, the bias $\left|\mathbb{E}[\lambda(\bsomega)] - \mathbb{E}[\lambda_{h,K}(\bsomega)]\right|$ is controlled by $h$ and $K$, whereas the variance term decreases linearly with $1/N$. To guarantee that the MSE remains below a threshold $\varepsilon^2$, $h$ and $K$ need to be chosen such that the bias is $O(\varepsilon^2)$, while the sample size needs to satisfy $N = O(\varepsilon^{-2})$. Suppose $K = K(h)$ is sufficiently large so that the bias is solely controlled by $h$ and satisfies $\left|\mathbb{E}[\lambda(\bsomega)] - \mathbb{E}[\lambda_{h,K}(\bsomega)]\right| = O(h^\alpha)$ for some $\alpha >0$. Suppose further that the computational cost to compute $\lambda_{h,K}(\bsomega)$ for each $\bsomega$ is $O(h^{-\gamma})$ for some $\gamma>0$. Then the total computational complexity to achieve an MSE of $\varepsilon^2$ is $O(\varepsilon^{-2-\gamma/\alpha})$. Note that in the best-case scenario, we have $\gamma = d$, i.e., when the computational cost of an eigensolver iteration is linear in the degrees of freedom of the discretization and the number of iterations can be bounded independently of $h$. Due to the quadratic convergence of algebraic eigensolvers, $K$ is usually controlled very easily. 

The multilevel Monte Carlo (MLMC) method offers a natural way to reduce the complexity of the standard Monte Carlo method by spreading the samples over a hierarchy of discretizations.
In our setting, we define a sequence of meshes corresponding to mesh sizes $h_0>h_1>\cdots>h_L > 0$.
This in turn defines a sequence of discretized eigenvalues $\lambda_{h_0, K_0}(\bsomega), \lambda_{h_1, K_1}(\bsomega),\dots,\lambda_{h_{L}, K_L}(\bsomega)$ that
approximate $\lambda(\bsomega)$ with increasing accuracy and increasing computational cost. The MLMC method approximates $\mathbb{E}[\lambda(\bsomega)]$ using the telescoping sum
\begin{equation}
\label{eq:telesum}
	\mathbb{E}[\lambda(\bsomega)] \approx \mathbb{E}[\lambda_L(\bsomega)] = \mathbb{E}[\lambda_0(\bsomega)] + \sum_{\ell = 1}^L\mathbb{E}[\lambda_\ell(\bsomega) - \lambda_{\ell - 1}(\bsomega)],
\end{equation}
where $\lambda_\ell(\bsomega):=\lambda_{h_\ell,K_\ell}(\bsomega)$ is the shorthand notation for the discretized eigenvalues. 
Each expected value of differences in \eqref{eq:telesum} can be estimated by an independent Monte Carlo approximation,
leading to the multilevel estimator 
\begin{equation}\label{eq:mlmc_eig}
Y=\sum_{\ell=0}^L Y_\ell,\mspace{15mu}Y_\ell = \frac{1}{N_\ell}\sum_{n=1}^{N_\ell}(\lambda_\ell(\bsomega_{\ell, n})-\lambda_{\ell-1}(\bsomega_{\ell, n})).
\end{equation}

Suppose independent samples are used to compute each $Y_\ell$, then
\begin{equation}
\mathbb{E}[Y]=\mathbb{E}[\lambda_L(\bsomega)],\mspace{10mu}\mathrm{var}[Y]=\sum_{\ell=0}^L \frac1{N_\ell} \mathrm{var}[\lambda_\ell(\bsomega)-\lambda_{\ell-1}(\bsomega)],
\end{equation}
and the MSE of \eqref{eq:mlmc_eig} can also be split into 
a bias and a variance term, i.e.,
\[
\mathrm{MSE}(\mathbb{E}[\lambda(\bsomega)], Y) = \left|\mathbb{E}[\lambda(\bsomega)] - \mathbb{E}[\lambda_{L}(\bsomega)]\right|^2 + \mathrm{var}(Y).
\]
Thus, to ensure again a MSE of $O(\varepsilon^2)$, it is sufficient to ensure that the
bias, $\left|\mathbb{E}[\lambda(\bsomega)] - \mathbb{E}[\lambda_{L}(\bsomega)]\right|^2 $, and the variance, $\mathrm{var}[Y]$, are both less than $\frac{1}{2}\varepsilon^2$.
The following theorem from \cite{Cliffe2011} (see also \cite{maingiles}) provides bounds on the computational cost of a general 
MLMC estimator and applies in particular to \eqref{eq:mlmc_eig}.

\begin{theorem}
Let $Q$ denote a random variable and $Q_{\ell}$ its numerical approximation on level $\ell$, and suppose $C_\ell$ is the computational cost of evaluating one realization of the difference $Q_\ell-Q_{\ell-1}$. Consider the multilevel estimator
\begin{equation}
Y=\sum_{\ell=0}^LY_\ell,\quad Y_\ell = \frac1{N_\ell}\sum_{n=1}^{N_\ell} Q_{\ell,n}-Q_{\ell-1,n},
\end{equation}
where $Q_{\ell,n}$ is a sample of $Q_{\ell}$ and $Q_{-1,n} = 0$, for all $n$. 

If there exist positive constants $\alpha, \beta, \gamma$ such that $\alpha\geq\frac{1}{2}\min(\beta,\gamma)$ and
\begin{itemize}[leftmargin=2em]
\item[I\hphantom{II}] \ $|\mathbb{E}[Q_\ell-Q]| = O( h_{\ell}^{\alpha})$ \; (convergence of bias),
\item[II\hphantom{I}] \ $\mathrm{var}[Y_\ell] = O(h_{\ell}^{\beta})$ \qquad\; (convergence of variance),
\item[III]  \ $C_\ell = O(h_{\ell}^{-\gamma})$ \qquad\quad\;\; (cost per sample),
\end{itemize}
then for any $0 < \varepsilon < e^{-1}$ there exist a constant $c$, a stopping level $L$, and sample sizes $\{N_\ell\}_{\ell = 0}^L$ such that the MSE of $Y$ satisfies $\mathrm{MSE}(\mathbb{E}[Q], Y) \leq \varepsilon^2$ with a total computational complexity, denoted by $C(\varepsilon)$, satisfying 
\begin{equation}
C(\varepsilon)\leq\left\{
  \begin{array}{ll}
    c\varepsilon^{-2}, & \hbox{$\beta > \gamma$;} \\
    c\varepsilon^{-2}(\log\varepsilon)^2, & \hbox{$\beta = \gamma$;} \\
    c\varepsilon^{-2-(\gamma-\beta)/\alpha}, & \hbox{$\beta < \gamma$,}
  \end{array}
\right.
	\label{eq:complexity}
\end{equation}
where the constant $c$ is independent of $\alpha$, $\beta$ and $\gamma$.
\label{thm:mlmc}
\end{theorem}

For a given $\varepsilon$, from \cite{Cliffe2011} the maximum level $L$ in Theorem~\ref{thm:mlmc} is given by
\begin{equation}
    \label{eq:L}
    L = \big\lceil \alpha^{-1} \log_2(\sqrt{2}\,c_I\,\varepsilon^{-1}) \big\rceil,
\end{equation}
where $c_I$ is the implicit constant from Assumption \emph{I} (convergence of bias) above.
The optimal sample sizes,
$\{N_\ell\}$,
that minimize the computational cost of the multilevel estimator in 
Theorem~\ref{thm:mlmc} are obtained using a standard Lagrange multipliers argument as in \cite{Cliffe2011} and are given by
\begin{equation}
\label{eq:N_ell}
N_\ell= \Bigg\lceil 2\varepsilon^{-2}\sqrt{\frac{\mathrm{var}[Q_\ell-Q_{\ell-1}]}{C_\ell}}
\sum^{L}_{i=0}\sqrt{\mathrm{var}[Q_i-Q_{i-1}]C_i} \Bigg\rceil, \quad \ell = 0, \ldots, L.
\end{equation}

Since $\beta > 0$, Theorem~\ref{thm:mlmc} 
shows that for
all cases in \eqref{eq:complexity},
the MLMC complexity is superior to that of Monte Carlo.
When $\beta>\gamma$, the variance reduction rate is larger than the rate of increase of the computational cost, and thus most of the work is spent on the coarsest level. In this case, the multilevel estimator has the best computational complexity. When $\beta<\gamma$ the total computational work of the multilevel estimator may only have a marginal improvement compared to that of the classic Monte Carlo method. 

\begin{corollary}[Order of convergence]
\label{cor:conv}
For $\bsomega \in \Omega$, let $h > 0$ be sufficiently small and consider two finite element approximations, cf. (\ref{eq:evp-fe}), of the smallest eigenvalue $\lambda(\bsomega)$ of the eigenvalue problem \eqref{eq:varevp} with $h_{\ell-1} = h$ and $h_{\ell} = h/2$.
The expectation of their difference is bounded by
\begin{equation}
\big|\mathbb{E}[\lambda_\ell(\bsomega)-\lambda_{\ell-1}(\bsomega)]\big|\leq c_1h_\ell^{2},
\end{equation}
while the variance of the difference is bounded by
\begin{equation}
\mathrm{var}[\lambda_\ell(\bsomega)-\lambda_{\ell-1}(\bsomega)]\leq c_2h_\ell^{4},
\end{equation}
for two constants $c_1, c_2$ that are independent of $\bsomega$, $h$ and $\ell$.
\end{corollary}
\begin{proof}
Applying Theorem~\ref{thm:fe}, since $C_\lambda$ is independent of $\bsomega$ we have
\begin{equation}
\big|\mathbb{E}[\lambda(\bsomega)-\lambda_\ell(\bsomega)]\big|
\leq
\mathbb{E}[|\lambda(\bsomega)-\lambda_\ell(\bsomega)|]\leq C_\lambda \Big(\frac{h}{2}\Big)^2,
\end{equation}
and
\begin{equation}
\big|\mathbb{E}[\lambda(\bsomega)-\lambda_{\ell-1}(\bsomega)]\big|
\leq
\mathbb{E}[|\lambda(\bsomega)-\lambda_{\ell-1}(\bsomega)|]\leq C_\lambda h^2.
\end{equation}
Therefore, by the triangle inequality, we have
\begin{equation}
\begin{aligned}
\big|\mathbb{E}[\lambda_\ell(\bsomega)-\lambda_{\ell-1}(\bsomega)]\big|
&=\big|\mathbb{E}[\lambda_\ell(\bsomega)-\lambda(\bsomega)+\lambda(\bsomega)-\lambda_{\ell-1}(\bsomega)]\big|
\\
&\leq\mathbb{E}[|\lambda(\bsomega)-\lambda_\ell(\bsomega)|]+\mathbb{E}[|\lambda(\bsomega)-\lambda_{\ell-1}(\bsomega)|]\\
&\leq C_\lambda \Big(h^2+\frac{h^2}{4}\Big)
=5 C_\lambda h_\ell^{2}.
\end{aligned}
\end{equation}

The variance reduction rate comes from the following relation
\begin{equation}
\mathrm{var}[\lambda(\bsomega)-\lambda_\ell(\bsomega)]\leq\mathbb{E}[(\lambda(\bsomega)-\lambda_\ell(\bsomega))^2]\leq C_\lambda^2\Big(\frac{h}{2}\Big)^4,
\end{equation}
and, similarly, by the Cauchy-Schwarz inequality 
\begin{align*}
\mathrm{var}[\lambda_\ell(\bsomega)-\lambda_{\ell-1}(\bsomega)]& \leq\mathbb{E}[(\lambda_\ell(\bsomega)-\lambda_{\ell-1}(\bsomega))^2] \\
& = \mathbb{E}[(\lambda_\ell(\bsomega)-\lambda(\bsomega) +\lambda(\bsomega) - \lambda_{\ell-1}(\bsomega))^2] \\\
& \leq 2 \big( \mathbb{E}[(\lambda(\bsomega)-\lambda_\ell(\bsomega))^2] +  \mathbb{E}[(\lambda(\bsomega) - \lambda_{\ell-1}(\bsomega))^2] \big) \\
& \leq 2 \Big( C_\lambda^2\Big(\frac{h}{2}\Big)^4 + C_\lambda^2h^4\Big) = 34 C_\lambda^2 h_\ell^4.
\end{align*}
\end{proof}

\begin{remark}
\label{rem:supg}
In our numerical experiments, we observed that the SUPG approximation of the eigenvalue problem, cf. (\ref{eq:evp-supg}), has similar rates of convergence $\alpha$ and $\beta$ in MLMC compared to the standard finite element approximation. 
\end{remark}

An important physical property of the smallest eigenvalue of 
\eqref{eq:varevp} is that it is real and strictly 
positive. Clearly, $\bbE[\lambda] > 0$ as well,
and so we would like our multilevel approximation 
\eqref{eq:mlmc_eig} to preserve this property.
Below we show that a multilevel approximation based on Galerkin FEM with
a geometrically-decreasing sequence of meshwidths
is strictly positive provided that $h_0$ is sufficiently small. 

\begin{proposition}
\label{prop:ML-pos}
Suppose that $h_\ell = h_0 2^{-\ell}$ for $\ell \in \N$ with $h_0 > 0$ sufficiently small and let $\lambda_{h_\ell}(\cdot)$ be the approximation of the smallest eigenvalue using the Galerkin FEM as in \eqref{eq:evp-fe}.
Then, for any $L \in \N$, the multilevel approximation 
of the smallest eigenvalue is strictly positive, i.e., 
\[
\widetilde{Y} \,\coloneqq\,
\sum_{\ell = 0}^L \widetilde{Y}_\ell
\,=\,
\sum_{\ell = 0}^L \frac{1}{N_\ell} \sum_{n = 1}^{N_\ell} 
\big(\lambda_{h_\ell}(\bsomega_{\ell, n}) - \lambda_{h_{\ell - 1}}(\bsomega_{\ell, n})\big)
\,>\, 0.
\]
\end{proposition}

\begin{proof}
First, since $\lambda$ is continuous and strictly positive on $\Omega$ it can be bounded uniformly from below, i.e., there
exists $\widecheck{\lambda} > 0$ such that
\begin{equation}
\label{eq:lambda-lower}
\lambda(\bsomega) \,\geq\, \widecheck{\lambda} > 0
\quad \text{for all } \bsomega \in \Omega.
\end{equation}

For $\ell = 0$, using 
\eqref{eq:lam-fe-err} and \eqref{eq:lambda-lower} we can bound $\lambda_{h_0}(\bsomega)$ 
uniformly from below by
\begin{align*}
\lambda_{h_0}(\bsomega) \,=\,
\lambda(\bsomega) - 
\big(\lambda(\bsomega) - \lambda_{h_0}(\bsomega)\big)
\,\geq\, \widecheck{\lambda}  - C_\lambda h_0^2\,.
\end{align*}
Since this bound is independent of $\bsomega$,
it follows that
\begin{equation}
\label{eq:Y0-lower}
\widetilde{Y}_0 \,\coloneqq\,
\frac{1}{N_0} \sum_{n = 1}^{N_0} 
\lambda_{h_0}(\bsomega_{0, n})
\,\geq\, \frac{1}{N_0} \sum_{n = 1}^{N_0} 
\big(\widecheck{\lambda}  - C_\lambda h_0^2\big)
\,=\, \widecheck{\lambda}  - C_\lambda h_0^2.
\end{equation}

Similarly, for $\ell \geq 1$ using \eqref{eq:lam-fe-err}
we obtain
\begin{align*}
\lambda_{h_\ell}(\bsomega) - \lambda_{h_{\ell - 1}}(\bsomega) 
\,&=\, \lambda(\bsomega) - \lambda_{h_{\ell - 1}} - 
\big(\lambda(\bsomega) - \lambda_{h_\ell}(\bsomega)\big)
\\
&\geq\, - \big|\lambda(\bsomega) - \lambda_{h_{\ell - 1}}\big| - 
\big|\lambda(\bsomega) - \lambda_{h_\ell}(\bsomega)\big|
\\
&\geq\, -C_\lambda\big( h_{\ell - 1}^2 + h_\ell^2\big)
\,=\, -9C_\lambda h_0^2 \,2^{-2\ell}\,.
\end{align*}
Again, this bound is independent of $\bsomega$
and so
\begin{equation}
\label{eq:Y_ell-lower}
\widetilde{Y}_\ell \,\coloneqq\, \frac{1}{N_\ell} \sum_{n = 1}^{N_\ell}
\big(\lambda_{h_\ell}(\bsomega_{\ell, n}) - \lambda_{h_{\ell - 1}}(\bsomega_{\ell, n})\big)
\,\geq\, -9C_\lambda h_0^2 \,2^{-2\ell}\,.
\end{equation}

Finally, we bound the multilevel approximation $\widetilde{Y}$ from below using \eqref{eq:Y0-lower}
and \eqref{eq:Y_ell-lower} as follows,
\begin{align*}
\widetilde{Y} \,&=\, \widetilde{Y}_0 + \sum_{\ell = 1}^L \widetilde{Y}_\ell
\,\geq\, \widecheck{\lambda}  - C_\lambda h_0^2
- \sum_{\ell = 1}^L 9C_\lambda h_0^2 \,2^{-2\ell}
\\
\,&>\,
\widecheck {\lambda} -9C_\lambda h_0^2 \sum_{\ell = 0}^L 2^{-2\ell}
\,>\,
\widecheck {\lambda} -9C_\lambda h_0^2 \sum_{\ell = 0}^\infty 2^{-2\ell}
\,=\, 
\widecheck {\lambda} -12C_\lambda h_0^2\,>\, 0,
\end{align*}
where we have used the property that $h_0$ is sufficiently small, i.e., $h_0 \leq \sqrt{\widecheck{\lambda}/(12C_\lambda)}$,
to ensure $\widetilde{Y} > 0$, as required.
\end{proof}

The result above can be extended beyond the geometric sequence of FE meshwidths to a general sequence of FE meshwidths, provided that $\sum_{\ell = 0}^L h_\ell^2$ is sufficiently small.
Similarly, as in Remark~\ref{rem:supg}, we observe that the 
MLMC approximations based on SUPG are also strictly positive.

Choosing the number of iterations $K_\ell$ such that the 
error of the eigensolver is of the same order as the FE error on 
each level, i.e., $|\lambda_{h_\ell}(\bsomega) - \lambda_{h_\ell, K_\ell}(\bsomega)| \lesssim h_\ell^2$ for all $\ell = 0, 1, \ldots, L$ and $\bsomega \in \Omega$, it can similarly be shown that the multilevel approximation \eqref{eq:mlmc_eig} also satisfies $Y > 0$.

To obtain the eigenvalue approximation on level $\ell$, choosing a basis for the FE space $V_\ell \coloneqq V_{h_\ell}$ in \eqref{eq:evp-fe} leads to a generalized (algebraic) eigenproblem in matrix form for each sample $\bsomega$, i.e.,
\begin{equation}
\label{eq:evp-matrix}
\bb{A_\ell}(\bsomega) \bb{u}_\ell(\bsomega) =\lambda_\ell(\bsomega) \bb{M_\ell}(\bsomega) \bb{u}_\ell(\bsomega),
\end{equation}
where $\bb{u}_\ell(\bsomega)$ is the coefficient vector (with respect to the basis) and $\bb{A}_\ell(\bsomega)$, $\bb{M}_\ell(\bsomega)$ are the associated FE matrices corresponding to the mesh $\mathcal{T}_\ell:=\mathcal{T}_{h_\ell}$.
The number of iterations $K$ in the computational cost per sample, as well as the rate of the cost per iteration 
depend on the choice of the algebraic eigensolver.
A variety of solvers can be applied here to solve the generalized eigenvalue problem \eqref{eq:evp-matrix}, including power iteration, the QR algorithm, subspace iterations, etc.
For our purposes, we only need an eigensolver that is able to compute the smallest eigenvalue, which is real and simple. As such, we consider here two eigenvalue solvers, the \emph{Rayleigh quotient iteration} and the \emph{implicitly restarted Arnoldi method}.

\begin{algorithm}[h!]
\caption{The Rayleigh quotient iteration (RQI).}
\begin{algorithmic}[1]
\STATE{Input: ($\mathbf{A}, \mathbf{M}, \bm{\eta}_0, \bm{\xi}_0, \lambda_0, \varepsilon, $M$)$, where $\bm\eta_0, \bm\xi_0, \lambda_0, \varepsilon$ and $M$ are initial left and right eigenvectors, the initial eigenvalue, the error tolerance, and the maximum number of iterations, respectively}
\STATE{Set $i\leftarrow 0$}
\WHILE{$\|\mathbf{A}\bm{\eta}_i-\lambda\mathbf{M}\bm{\eta}_i\|>\varepsilon$ and $i \leq M$}
	\STATE{Normalize $\bm\eta_{i}\leftarrow\bm{\eta}_{i}\|\bm\eta_{i}\|^{-1}_2$}
	\STATE{Normalize $\bm\xi_{i}\leftarrow \bm\xi_{i}\|\bm\xi_{i}\|_2^{-1}$}
	\STATE{Solve $(\lambda_{i}\mathbf{M}-\mathbf{A})\bm\eta_{i+1}=\bm\eta_{i}$}
	\STATE{Solve $(\lambda_{i}\mathbf{M}-\mathbf{A})^H\bm\xi_{i+1}=\bm\xi_{i}$}
	\STATE{Compute $\lambda_{i+1}\leftarrow (\bm{\xi}^H_{i+1}\mathbf{A}\bm{\eta}_{i+1})(\bm{\xi}^H_{i+1}\mathbf{M}\bm{\eta}_{i+1})^{-1}$}
	\STATE{$i\leftarrow i + 1$}
\ENDWHILE
\STATE{Output: ($\bm{\eta}, \bm{\xi}, \lambda)$}
\end{algorithmic}
\label{alg:rqi}
\end{algorithm}

We first consider the Rayleigh quotient iteration (Alg.~\ref{alg:rqi}), 
introduced first by Lord Rayleigh in 1894 for a quadratic eigenproblem of oscillations of a mechanical system \cite{Rayleigh1894} and then extended in the 1950s and 60s to non-symmetric generalized eigenproblems  \cite{Crandall1951,1957ArRMA}. 
The following lemma, whose proof can be found in Crandall~\cite{Crandall1951} and Ostrowski~\cite{1957ArRMA}, establishes the error reduction rate of the Rayleigh quotient iteration, which will in turn help to bound the computational cost on each level.
\begin{lemma}
	Suppose we have an initial guess $\lambda_{\ell,0}(\bsomega)$ to the eigenvalue $\lambda_\ell(\bsomega)$ at the level $\ell$ and $|\lambda_{\ell,0}(\bsomega)-\lambda_\ell(\bsomega)|$ is sufficiently small. Then the sequence $\lambda_{\ell,i}(\bsomega)$ converges to $\lambda_\ell(\bsomega)$ quadratically\label{lem:rq1}, i.e., there exists a constant $\hat C(\bsomega)$ such that
\begin{equation}
		|\lambda_{\ell,i+1}(\bsomega)-\lambda_\ell(\bsomega)|\leq \hat C(\bsomega)|\lambda_{\ell,i}(\bsomega)-\lambda_\ell(\bsomega)|^2.
\end{equation}
\end{lemma}

The computational cost of Rayleigh quotient iteration (RQI) is dominated by the cost of solving two linear systems in each iteration (cf.~Lines 6 and 7 of Alg.~\ref{alg:rqi}).
For direct solvers, such as LU decomposition, the computational cost depends on the sparsity and bandwidth of the matrices, e.g., for piecewise linear FE applied to \eqref{eq:varevp} and $d = 2$, the cost for solving these linear systems on level $\ell$ is $O(h_\ell^{-3})$ \cite{george1988complexity}. However, optimal iterative solvers, such as geometric multigrid methods, are able to achieve the optimal computational complexity of (or close to) $O(h_\ell^{-d})$.  All other steps in Alg.~\ref{alg:rqi} are linear in the degree of freedoms, and thus $O(h_\ell^{-d})$. Hence, typically the cost per iteration grows with rate $\gamma \ge d$, but it can be as big as $\gamma =3$ for $d=2$. The remaining factor in the computational cost is the number of iterations $K$ for the Rayleigh quotient iteration within the MLMC estimator, but this is independent of $h_\ell$. 

\begin{algorithm}[h!]
	\caption{Three-grid Rayleigh Quotient iteration (tgRQI).}
\begin{algorithmic}[1]
\STATE{Input: ($\mathbf{A}_\ell, \mathbf{A}_{\ell-1}, \mathbf{A}_0, \mathbf{M}_\ell, \mathbf{M}_{\ell-1}, \mathbf{M}_0, \bm\eta_0,\bm\xi_0,\lambda_0, \ell$), where $\bm\eta_0', \bm\xi_0', \lambda_0'$ are the initial left and right eigenvectors at level $0$, and the initial eigenvalue.}
\STATE{$\varepsilon\leftarrow10^{-10}$, $M\leftarrow1000$}
\STATE{$(\bm\eta_0,\bm\xi_0,\lambda_0)\leftarrow$ RQI ($\mathbf{A}_{0}, \mathbf{M}_{0}, \bm\eta_0',\bm\xi_0',\lambda_0',\varepsilon,M$)}
\STATE{Interpolate the eigenfunctions from $V_0$ on $\mathcal{T}_0$ onto $V_{\ell-1}$ on $\mathcal{T}_{\ell-1}$:\\ \qquad $(\bm\eta_{\ell-1}',\bm\xi_{\ell-1}')\leftarrow(\bm\eta_0,\bm\xi_0)$}
\STATE{$(\bm\eta_{\ell-1},\bm\xi_{\ell-1},\lambda_{\ell-1})\leftarrow$ RQI ($\mathbf{A}_{\ell-1}, \mathbf{M}_{\ell-1}\bm\eta_{\ell-1}',\bm\xi_{\ell-1}',\lambda_{0},\varepsilon,M$)}
\IF{$\ell-1 = 0$}
\STATE{Output: $\lambda_1-\lambda_0$}
\ELSE
\STATE{Interpolate the eigenfunctions from $V_{\ell-1}$ on $\mathcal{T}_{\ell-1}$ onto $V_{\ell}$ on $\mathcal{T}_{\ell}$: $(\bm\eta_{\ell}',\bm\xi_{\ell}')\leftarrow(\bm\eta_{\ell-1},\bm\xi_{\ell-1})$}
\STATE{$(\bm\eta_\ell,\bm\xi_\ell,\lambda_\ell)\leftarrow$ RQI ($\mathbf{A}_\ell, \mathbf{M}_\ell, \bm\eta_\ell',\bm\xi_\ell',\lambda_{\ell-1},\varepsilon,M$)}
\STATE{Output: $\lambda_\ell - \lambda_{\ell-1}$}
\ENDIF
	\end{algorithmic}
	\label{alg:threegrid}
\end{algorithm}

Recall the MLMC estimator \eqref{eq:mlmc_eig}, where at each level $\ell$ we compute the differences $\lambda_{\ell}(\bsomega_n)-\lambda_{\ell-1}(\bsomega_n)$ for the same sample $\bsomega_n$. The number of RQI iterations needed for a sufficiently accurate approximation of $\lambda_{\ell}(\bsomega_n)$ -- the more costly level $\ell$ computation --  
 can be significantly reduced by using the computed approximation of the eigenvalue 
 $\lambda_{\ell-1}(\bsomega_n)$ on the coarser level as the initial guess, thus also reducing the total computational cost. In fact, we design a three-grid method, similar to the one used in \cite{Gilbert2021} to implement this strategy, which uses the approximate eigenvalue $\lambda_{0}(\bsomega_n)$ on level zero with mesh size $h_0$ as the initial guess for computing eigenvalue $\lambda_{\ell-1}(\bsomega_n)$ on level $\ell-1$. Then, $\lambda_{\ell-1}(\bsomega_n)$ is used as the initial guess for computing $\lambda_{\ell}(\bsomega_n)$; see Alg.~\ref{alg:threegrid} for details.

To estimate the computational cost of this three-grid method, we choose again $h_{\ell-1} = h = 2 h_{\ell}$ and denote the exact discrete eigenvalues on level $\ell-1$ and level $\ell$ by $\lambda_{h}(\bsomega_n)$ and $\lambda_{h/2}(\bsomega_n)$, respectively. The goal is to control the errors of the eigenvalues $\lambda_{\ell-1}(\bsomega_n)$ and $\lambda_{\ell}(\bsomega_n)$ actually computed using Alg.~\ref{alg:threegrid} to be within the respective discretization errors. 
Due to the quadratic convergence rate of the RQI (cf. Lemma \ref{lem:rq1}), often only two or three iterations are sufficient to compute a sufficiently accurate approximation $\lambda_{0}(\bsomega_n)$ on Level 0 in Line 3 of Alg.~\ref{alg:threegrid}. Similarly, in Line 5 of Alg.~\ref{alg:threegrid}, two to three iterations of RQI are again sufficient to ensure that the error of the estimated eigenvalue $\lambda_{\ell-1}(\bsomega_n)$ satisfies
\[
| \lambda_{\ell-1}(\bsomega_n) - \lambda_h(\bsomega_n) |
\leq C_\lambda h_{\ell-1}^2, 
\]
which is the bound on the discretization error on level $\ell -1$ in Theorem \ref{thm:fe}. When $\lambda_{\ell-1}(\bsomega_n)$ is then used as the initial guess for estimating $\lambda_{h/2}(\bsomega_n)$, the initial error satisfies
\[
| \lambda_{\ell-1}(\bsomega_n) - \lambda_{h/2}(\bsomega_n) | \leq | \lambda_{\ell-1}(\bsomega_n) - \lambda_h(\bsomega_n) | + |\lambda_h(\bsomega_n) - \lambda_{h/2}(\bsomega_n)|\leq \frac94 C_\lambda h^2,
\]
using triangle inequality and Theorem \ref{thm:fe} again. Therefore, using Lemma \ref{lem:rq1} for sufficiently small mesh size $h$ such that $h \leq \frac29 \big(\hat C(\bsomega_n) C_\lambda \big)^{-1/2}$, 
one single iteration of RQI on level $\ell$ suffices such that
\[
| \lambda_{\ell}(\bsomega_n) - \lambda_{h/2}(\bsomega_n) |
\leq C_\lambda h_{\ell}^2.
\]

In practice, two iterations of RQI are typically used to achieve the target accuracy for $\lambda_\ell(\bsomega_n)$ in Line 10 of Alg.~\ref{alg:threegrid}. These two calls to RQI dominate the computational cost of Alg.~\ref{alg:threegrid} with their four linear solves. Hence, for sparse direct solvers and $d=2$, the overall computational cost of Alg.~\ref{alg:threegrid} is $O(h_\ell^{-3})$ and $\gamma = 3$ in Theorem \ref{thm:mlmc}.
The computational complexity of Alg.~\ref{alg:threegrid} can be further reduced using multigrid-based methods to efficiently solve the Rayleigh quotient iterations~\cite{Cai1997} that potentially offer a rate of $\gamma = d$ (or close to) even in three dimensions. However, it is unclear if the same rate of convergence as for self-adjoint operators can be retained for the convection-dominated problems we are considering here.

\begin{algorithm}
\caption{Multilevel Monte Carlo algorithm.}
\begin{algorithmic}[1]
\FOR{$i=1\ldots N_0$}
\STATE{Draw a sample $\bsomega_i$}
\STATE{Compute $\lambda_0(\bsomega_i)$ using either Alg.~\ref{alg:rqi} or ARPACK }
\ENDFOR
\FOR{$\ell=1\ldots L$}
\FOR{$i=1\ldots N_\ell$}
\STATE{Draw a sample $\bsomega_i$}
\STATE{Compute $\lambda_\ell(\bsomega_i)-\lambda_{\ell-1}(\bsomega_i)$ using either Alg.~\ref{alg:threegrid} or ARPACK}
\ENDFOR
\ENDFOR
\end{algorithmic}
\label{alg:mlmc}
\end{algorithm}

We also consider the implicitly restarted Arnoldi method \cite{Arnoldi1951,Lehoucq1995,Saad1980,Saad1984,Scott1995}
and its implementation in the library ARPACK \cite{ARPACK} to solve the eigenvalue problem. 
Compared to the Rayleigh quotient iteration, the Arnoldi method calculates a specified number of eigenpairs that depend on the dimension of the Krylov subspace.
The performance of the implicitly restarted Arnoldi method is determined by several factors such as the dimension of the Krylov subspace and the initial vector. To the best of the authors' knowledge, for the eigenvalue problem \eqref{eq:evp-fe} we are considering here, the convergence rate, and therefore the computational cost, of the implicitly restarted Arnoldi method is not yet known. As such, we numerically estimate the rate variable $\gamma$ and the computational cost $C_\ell$ for determining the optimal sample sizes in MLMC. It appears that the number of iterations grows slightly faster than  $O(h_\ell^{-1})$ leading to a similar total complexity as RQI for $d=2$ of $\gamma \approx 3.5$.

\section{Extensions of MLMC method}
\label{sec:ext_mlmc}
In this section, we introduce two extensions of the MLMC method for convection-diffusion eigenvalue problems.
First, we employ a homotopy method to add stability to the eigensolve for each sample.
Second, we replace the Monte Carlo approximation of the expected value on each level in \eqref{eq:telesum}
with a quasi-Monte Carlo (QMC) method, which, due to the faster convergence of QMC, 
allows us to use less samples on each level and improves the overall complexity.

\subsection{Homotopy multilevel Monte Carlo method}
\label{sec:homotopy}

In Carstensen et al.~\cite{homo2011} a homotopy method is employed to solve convection-diffusion eigenvalue problems with deterministic coefficients, using the homotopy method to derive adaptation strategies for FE methods.
The authors also provided estimates on the convergence rate of the smallest eigenvalue with respect to the homotopy parameter.
We aim to investigate the application of this homotopy method in the MLMC method, particularly in designing multilevel models for alleviating numerical instability (due to the high advection velocity) on coarser meshes.

For eigenvalue problems, the homotopy method~\cite{homo1997} uses an initial operator $\mathcal{L}_0$---for which the target eigenvalue is easier to compute than that of the original operator $\mathcal{L}$---to form a continuation \begin{equation}
	\mathcal{L}_t = (1-f(t))\mathcal{L}_0+f(t)\mathcal{L} \mspace{10mu} \text{for} \mspace{10mu} 0\leq t\leq 1,
	\label{eq:homo}
\end{equation}
with a function $f:[0;1]\rightarrow[0;1]$ and $f(0)=0$, $f(1)=1$. 
For the convection-diffusion operator in~(\ref{convdiffeq}), it is natural to set the diffusion operator as the initial operator. Here we consider a simple linear function $f(t)=t$ to design the sequence of operators used for the homotopy. Given a sequence of homotopy parameters, $0 = t_0 < t_1 < \cdots < t_L = 1$, the homotopy operators with stochastic coefficients define a sequence of eigenvalue problems of the form 
\begin{align}\label{eq:homo_eig}
	\mathcal{H}(\bsomega,t_\ell) u(\bsomega,t_\ell) & =-\nabla\cdot\big(\kappa(\bsomega)\nabla u(\bsomega,t_\ell)\big) + t_\ell\big(\mathbf{a}(\bsomega)\cdot\nabla u(\bsomega,t_\ell)\big) \nonumber \\
	& = \lambda(\bsomega,t_\ell) u(\bsomega,t_\ell),
\end{align}
for $\ell = 0, \ldots, L$. The following lemma \cite[Lemma 4.1]{homo2011} establishes the homotopy error on the smallest eigenvalue in \eqref{eq:homo_eig} for fixed $\bsomega$.
\begin{lemma}\label{lem:homotop}
Suppose the velocity field $\bb{a}$ is divergence-free and $\bsomega$ is fixed. The homotopy error---which is defined as the difference between the smallest eigenvalue $\lambda(\bsomega,t=1)$ of the original operator and that of the homotopy operator in \eqref{eq:homo_eig} satisfies for any $t\in[0, 1]$ 
\begin{equation}
\label{eq:homo-err}
|\lambda(\bsomega,1)-\lambda(\bsomega,t)|
\,\leq\, C_{t, \bsomega} (1-t),
\end{equation}
where
\begin{equation}
\label{eq:C_homo}
C_{t, \bsomega} \coloneqq \frac{\|\bsa(\cdot, \bsomega)\|_{L^\infty}
\big(\|u(\bsomega, 1)\|_V + \|u^*(\bsomega, 1)\|_V\big)}
{\langle u(\bsomega, 1), u^*(\bsomega, t)\rangle + \langle u(\bsomega, t), u^*(\bsomega, 1)\rangle},
\end{equation}
and $u^*(\bsomega,t)$ is the dual homotopy solution.
For $t$ sufficiently close to 1 and almost all $\bsomega \in \Omega$, 
$C_{t, \bsomega} < C_t $
for some $C_t < \infty$ independent of $\bsomega$.
\end{lemma}

\begin{proof}
First, the primal and dual homotopy eigenvalue problems are
\begin{align*}
\calA(\bsomega; u(\bsomega, t), v) + t\calB(u(\bsomega, t), v) &= \lambda(\bsomega, t)
\langle u(\bsomega, t), v\rangle \quad &\text{for all } v \in V,\\
\calA(\bsomega; w, u^*(\bsomega, t)) + t\calB(w, u^*(\bsomega, t)) &= \overline{\lambda^*(\bsomega, t)}
\langle w, u^*(\bsomega, t)\rangle \quad &\text{for all } w \in V,
\end{align*}
where we again normalise the homotopy eigenfunctions such that
$\|u(\bsomega, t)\|_{L^2} = 1 = \|u^*(\bsomega, t)\|_{L^2}$.

Following the proof of \cite[Lemma 4.1]{homo2011}, 
using the homotopy eigenvalue problems we can write the homotopy error as
\begin{align}
\label{eq:homo1}
&\big[\lambda(\bsomega, 1) - \lambda(\bsomega, t)\big]
\big[\langle u(\bsomega, 1), u^*(\bsomega, t)\rangle
+ \langle u(\bsomega, t), u^*(\bsomega, 1)\rangle\big]
\nonumber\\
=& \lambda(\bsomega, 1)\langle u(\bsomega, 1), u^*(\bsomega, t)\rangle
+ \overline{\lambda^*(\bsomega, 1)}\langle u(\bsomega, t), u^*(\bsomega, 1) \rangle
\nonumber\\
&- \overline{\lambda^*(\bsomega, t)}\langle u(\bsomega, 1), u^*(\bsomega, t)\rangle
- \lambda(\bsomega, t)\langle u(\bsomega, t), u^*(\bsomega, 1) \rangle
\nonumber\\
=& (1 - t)\big[\calB(\bsomega; u(\bsomega, 1), u^*(\bsomega, t)) + 
\calB(\bsomega; u(\bsomega, t), u^*(\bsomega, 1))\big],
\end{align}
where we have also used the property $\lambda(\bsomega, t) = \overline{\lambda^*(\bsomega, t)}$.

Since $\bsa(\bsomega)$ is divergence free, we have
\[
\calB(\bsomega; u(\bsomega, t), u^*(\bsomega, 1))
=
-\calB(\bsomega; \overline{u^*(\bsomega, 1)}, \overline{u(\bsomega, t)}).
\]
Then by the triangle inequality, followed by the Cauchy--Schwarz inequality
\begin{align}
\label{eq:homo2} 
&\calB(\bsomega; u(\bsomega, 1), u^*(\bsomega, t)) + 
\calB(\bsomega; u(\bsomega, t), u^*(\bsomega, 1))
\nonumber\\
&= \calB(\bsomega; u(\bsomega, 1), u^*(\bsomega, t)) 
-\calB(\bsomega; \overline{u^*(\bsomega, 1)}, \overline{u(\bsomega, t)})
\nonumber\\
&\leq |\calB(\bsomega; u(\bsomega, 1), u^*(\bsomega, t))|
+|\calB(\bsomega; \overline{u^*(\bsomega, 1)}, \overline{u(\bsomega, t)})|
\nonumber\\
&\leq \|\bsa(\bsomega)\|_{L^\infty} \|\nabla u(\bsomega, 1)\|_{L^2} \|u^*(\bsomega, t)\|_{L^2}
+ \|\bsa(\bsomega)\|_{L^\infty} \|\nabla u^*(\bsomega, 1)\|_{L^2} \|u(\bsomega, t)\|_{L^2}
\nonumber\\
&= \bsa_{\max} \big(\|u(\bsomega, 1)\|_{V} + \|u^*(\bsomega, 1)\|_V\big),
\end{align}
where we have used the property that the homotopy eigenfunctions are normalized
and Assumption~\ref{ass:3}.
Substituting \eqref{eq:homo2} into \eqref{eq:homo1} then rearranging gives the 
result \eqref{eq:homo-err} with $C_{t, \bsomega}$ as in \eqref{eq:C_homo}.

Next, we bound $C_{t, \bsomega}$ independently of $\bsomega$. 
Clearly, the numerator is bounded for all $t$ and almost all $\bsomega$.
Next, we show that the denominator is strictly positive.
Suppose for a contradiction that $\langle u(\bsomega, 1), u^*(\bsomega, t)\rangle = 0$,
then this implies that
\[
\langle u(\bsomega, 1), u^*(\bsomega, 1) - u^*(\bsomega, t)\rangle = 
\langle u(\bsomega, 1), u^*(\bsomega, 1)\rangle > 0,
\]
since the eigenfunction and dual eigenfunction are not orthogonal if the corresponding eigenvalues satisfy $\lambda(\bsomega, 1) = \overline{\lambda^*}(\bsomega, 1)$.
However, since $u^*(\bsomega, t) \to u^*(\bsomega, 1)$ as $t \to 1$,
the left hand side tends to zero whereas the right hand side is strictly positive and independent of $t$,
leading to a contradiction. Hence, for $t$ sufficiently small  $\langle u(\bsomega, 1), u^*(\bsomega, t)\rangle > 0$ and similarly $\langle u(\bsomega, t), u^*(\bsomega, 1)\rangle > 0$.
Thus, for $t$ sufficiently small $C_{t, \bsomega} < \infty$. Since $\bsa(\bsomega)$ along with the primal and dual eigenfunctions are continuous in $\bsomega$, it follows that $C_{t, \bsomega}$ is also continuous in $\bsomega$ and thus, can be bounded by the maximum over the compact domain $\Omega$,
\[
C_{t, \bsomega} \leq \max_{\bsomega \in \Omega} C_{t, \bsomega} \eqqcolon C_t < \infty.
\]
\end{proof}

With the homotopy method, the approximation error now comes from three sources: the FE discretization, the iterative eigensolver, and the value of the homotopy parameter. We suppose again that the error due to the eigensolver is bounded from above by the other two sources of error and design multilevel sequences such that the homotopy error and the discretization error are non-increasing with increasing level. Denoting the homotopy parameter 
and the mesh size at level $\ell$ by $t_\ell$ and $h_\ell$, respectively, the multilevel sequence
$$\{(t_0,h_0),(t_1,h_1),\ldots,(t_L,h_L)\},$$ 
is designed such that $t_{\ell-1} \leq t_\ell$, $h_{\ell-1} \geq h_\ell$, and $t_L=1$. The multilevel parameters are required to be non-repetitive, i.e., $(t_{\ell-1}, h_{\ell-1})\neq(t_\ell,h_\ell)$ for all $\ell = 1, \ldots, L$, to ensure an asymptotically decreasing total approximation error in the sequence. However, one of these two parameters is allowed to be the same on two adjacent levels, i.e., either $h_{\ell-1}=h_\ell$ or $t_{\ell-1}=t_\ell$ is possible. This setting allows for adapting the homotopy parameter to discretisations on different meshes to satisfy the stability condition of the FE approximation.

The resulting MLMC estimator can be derived from the telescoping sum
\[
\mathbb{E}[\lambda(\bsomega)] = \mathbb{E}[\lambda_{h_0}(\bsomega,t_0)]+\sum_{i=1}^L\mathbb{E}[\lambda_{h_i}(\bsomega, t_i)-\lambda_{h_i-1}(\bsomega, t_{i-1})].
\]
Following a similar derivation as that of Corollary \ref{cor:conv} and based on the error bound in Lemma \ref{lem:homotop}, we conjecture that the expectation and the variance of the multilevel difference with the homotopy method are bounded by
\begin{equation}
\label{eq:homotopy}
\begin{aligned}
|\mathbb{E}[\lambda_{h_{\ell}}(\bsomega, t_{\ell})-\lambda_{h_{\ell-1}}(\bsomega, t_{\ell-1})]| &\leq c_1 h_{\ell-1}^2 + c_2 (1-t_{\ell-1}), 
\\
\mathrm{var}[\lambda_{h_{\ell}}(\bsomega, t_{\ell})-\lambda_{h_{\ell-1}}(\bsomega, t_{\ell-1})] &\leq c_3 h_{\ell-1}^4 + c_4 (1-t_{\ell-1})^2,
\end{aligned}
\end{equation}
respectively.
This will be used as the guideline for choosing the multilevel sequences in our numerical experiments. We will also demonstrate that the above conjecture is valid in our numerical experiments. 

\subsection{Multilevel QMC Methods}
QMC methods are a class of equal-weight quadrature rules
originally designed to approximate high-dimensional integrals on the unit hypercube.
A QMC approximation of the expected value of $f$ is given by
\begin{equation}
\label{eq:qmc}
    \bbE[f] \,=\, \int_{[0, 1]^s} f(\bsomega) \, \rd \bsomega \, \approx\, 
    \frac{1}{N} \sum_{k = 1}^{N - 1} f(\bstau_{k}),
\end{equation}
where, in contrast to Monte Carlo methods, the quadrature points $\{\bstau_k\}_{k = 1}^{N - 1} \subset [0, 1]^s$
are chosen deterministically to be well-distributed and have good approximation properties
in high dimensions.
There are several types of QMC methods, including lattice rules, digital nets and randomised rules. The main benefit of QMC methods is that for sufficiently smooth integrands the quadrature 
error converges at a rate of $\mathcal{O}(N^{-1 + \delta})$, $\delta > 0$, or faster, which is 
better than the Monte Carlo convergence rate of $\mathcal{O}(N^{-1/2})$.
For further details see, e.g., \cite{DKS13,DP10}.

In this paper, we consider randomly shifted lattice rules, which are generated by
a single integer vector $\bsz \in \N^s$ and a single random shift 
$\bsDelta \sim \mathrm{Uni}[0, 1]^s$. 
The points are given by
\begin{equation}
\label{eq:lattice}
    \bstau_k \,=\, \bigg\{ \frac{k\bsz}{N} + \bsDelta\bigg\} 
    \quad \text{for } k = 0, 1, \ldots, N - 1,
\end{equation}
where $\{\cdot\}$ denotes taking the fractional part of each component.
The benefits of random shifting are that the resulting approximation \eqref{eq:qmc} is unbiased
and that performing multiple QMC with i.i.d. random shifts provides a practical estimate
for the mean-square error using the sample variance of the multiple approximations.

If $f$ is sufficiently smooth (i.e., has square-integrable mixed first derivatives) then a generating vector can be constructed such that the mean-square error (MSE) of a randomly shifted lattice rule approximation satisfies
\begin{equation}
    \bbE \bigg[\bigg|  \int_{[0, 1]^s} f(\bsomega) \, \rd \bsomega - \frac{1}{N} \sum_{k = 0}^{N - 1} f(\bstau_{k})\bigg|^2\bigg]
    \lesssim N^{-1/\eta}
    \quad \text{for } \eta \in (\tfrac{1}{2}, 1],
\end{equation}
see, e.g., Theorem~5.10 in \cite{DKS13}. I.e., for $\eta \approx 1/2$ the convergence of the MSE is close to $1/N^{2}$.

Starting again with the telescoping sum \eqref{eq:telesum},
a multilevel QMC (MLQMC) method approximates the 
expectation of the smallest eigenvalue by using a QMC rule
to compute the expectation on each level.
MLQMC methods were first introduced in \cite{GilesWater09}
for SDEs,
then applied to parametric PDEs in \cite{KSSSU17,KSS15}
and elliptic eigenvalue problems in  \cite{Gilbert2020,Gilbert2021}.
For $L \in \N$ and $\{N_\ell\}_{\ell = 0}^L$,
the MLQMC approximation is given by
\begin{equation}
\label{eq:mlqmc}
    Y^\mathrm{MLQMC} \coloneqq \sum_{\ell = 0}^L
    Y^\mathrm{QMC}_\ell,\quad
    Y^\mathrm{QMC}_\ell \coloneqq \frac{1}{N_\ell}
    \sum_{\ell = 0}^{N_\ell - 1}
    \big[\lambda_\ell(\bstau_{\ell, k}) - \lambda_{\ell - 1}(\bstau_{\ell, k})\big],
\end{equation}
where we apply a different QMC rule with points $\{\bstau_{\ell, k}\}_{k = 0}^{N_\ell - 1}$ on each level,
e.g., an $N_\ell$-point randomly shifted lattice rule \eqref{eq:lattice} generated by $\bsz_\ell$ and an i.i.d. $\bsDelta_\ell$.

The faster convergence of QMC rules leads to an improved 
complexity of MLQMC methods compared to MLMC, where in the best case the cost is reduced to close to $\varepsilon^{-1}$
for a MSE of $\varepsilon^2$. 
Following \cite{KSSSU17}, under the same assumptions as in Theorem~\ref{thm:mlmc},
but with Assumption \emph{II} replaced by
\begin{itemize}[leftmargin=3em]
\item[\emph{II(b)}]\label{asm:IIb} $\mathrm{MSE}[Y_\ell^\mathrm{QMC}] = O(N_\ell^{-1/\eta} h_\ell^\beta)$ with $\eta \in (\frac{1}{2}, 1]$,
\end{itemize}
the MSE of the MLQMC estimator \eqref{eq:mlqmc} 
is bounded above by $\varepsilon^2$ and the cost satisfies
\begin{align*}
    C_\mathrm{MLQMC}(\varepsilon)
    \lesssim 
    \begin{cases}
        \varepsilon^{-2\eta} & \text{if } \beta\eta  > \gamma,\\
        \varepsilon^{-2\eta}\log_2(\varepsilon^{-1})^{\eta + 1} & \text{if } \beta \eta  = \gamma,\\
        \varepsilon^{-2\eta - (\gamma - \beta \eta)/\alpha} & 
        \text{if } \beta\eta  < \gamma.
    \end{cases}
\end{align*}
The maximum level $L$ is again given by \eqref{eq:L}
and $\{N_\ell\}$ are given by
\begin{equation}
\label{eq:N_ell_mlqmc}
    N_\ell \,=\, \Bigg\lceil N_0
    \bigg(\frac{h_\ell^\beta}{C_\ell}\bigg)^{\eta/(\eta + 1)}
    {C_0}\bigg]^{1/(\eta + 1)}\bigg)^\eta
    \Bigg\rceil\,,
\end{equation}
where $C_\ell$ is the cost per sample as in assumption 
\emph{III} in Theorem~\ref{alg:mlmc} and $N_0$ is chosen as
\begin{equation*}
    N_0 \simeq \varepsilon^{-2\eta} \Bigg(\sum_{\ell = 0}^L \big(h_\ell^{\beta\eta} C_\ell
    \big)^{1/(\eta + 1)}\Bigg)^\eta.
\end{equation*}

Verifying Assumption~\emph{II(b)} for the convection-diffusion EVP \eqref{convdiffeq} requires performing
a technical analysis similar to \cite{Gilbert2020} and in particular, requires bounding the derivatives of the eigenvalue $\lambda(\bsomega)$ and its eigenfunction $u(\bsomega)$ with respect to $\bsomega$.
Such analysis is left for future work.
In the numerical results, section we study the convergence of QMC and observe that \emph{II(b)} holds
with $\eta \approx 0.61$.

In practice, one should perform multiple, say $R \in \N_0$, QMC approximations corresponding to i.i.d. random shifts,
then take the average as the final estimate. In this way, we can also estimate the MSE by the sample variance over the different realisations.


\section{Numerical results}
\label{sec:nums}
In this section, we present numerical results for three test cases. The quantity of interest in all cases is the smallest eigenvalue of the stochastic convection-diffusion problem (\ref{convdiffeq}) in the unit domain $D=[0,1]^2$. The first two test cases use constant convection velocities at different magnitudes to benchmark the performance of eigenvalue solvers and finite element discretisation methods in the multilevel setting. In these two test cases, the random conductivity $\kappa(\mathbf{x};\bsomega)$ is modelled as a log-uniform random field constructed through the convolution of $s_\kappa$ i.i.d. uniform random variables
\[
\log\kappa(\mathbf{x};\bsomega)=\sum_{i=1}^{s_\kappa}\omega_ik(\mathbf{x}-\mathbf{c}_i),
\]
with exponential kernels $k(\mathbf{x}-\mathbf{c}_i)=\exp[-\frac{25}{2}\|\mathbf{x}-\mathbf{c}_i\|_2]$, where $\mathbf{c}_i$ are the kernel centers placed uniformly on a $5\times5$ grid in the domain $D$.
In the third test case, we also make the convection velocity a random field. Specifically, we first construct a log-uniform random field 
\begin{equation}\label{eq:random_vel1}
S(\bsomega,\bsx)=\exp\left[\sum_{i=1}^{s_a}\omega_{i + s_\kappa} k(\mathbf{x}-\mathbf{c}_i)\right],
\end{equation}
similar to that of the conductivity field using additional $s_a$ i.i.d. uniform random variables. Then, a divergence-free velocity field can be obtained by 
\begin{equation}\label{eq:random_vel2}
\bsa(\bsomega)=\left[\frac{\partial S(\bsomega,\bsx)}{\partial x_2}, -\frac{\partial S(\bsomega,\bsx)}{\partial x_1}\right]^\top.
\end{equation}

We employ the Eigen~\cite{Eigen} library for Rayleigh quotient iteration and solve the linear systems using sparse LU decomposition with permutation from the SuiteSparse~\cite{SuiteSparse} library. For the implicitly restarted Arnoldi method, we use the ARPACK~\cite{ARPACK} library with the \texttt{SM} mode for finding the smallest eigenvalue. Random variables are generated using the standard C++ library and the pseudo-random seeds are the same across all experiments.

Numerical experiments are organized as follows. 
For a relatively low convection velocity $\bb{a}=[20;0]^T$, we demonstrate the multilevel Monte Carlo (MLMC) method using the Galerkin FEM discretization. In this case, we also consider applying the homotopy method together with a geometrically refined mesh hierarchy. Then, on a test case with relatively high convection velocity $\bb{a}=[50;0]^T$, we demonstrate the extra efficiency gain offered by the numerically more stable SUPG method, compared with the Galerkin discretization. For the third test case with a random velocity field, we apply SUPG to demonstrate the efficacy and efficiency of our multilevel method. Here we also demonstrate that quasi-Monte Carlo (QMC) samples can be used to replace Monte Carlo samples to further enhance the efficiency of multilevel methods. For all multilevel methods, we consider a sequence of geometrically refined meshes with $h_\ell = h_0 \times 2^{-\ell}, \ell = 0, 1, \ldots, 4$, and $h_0 = 2^{-3}$. At the finest level, this gives 16129 degrees of freedom in the discretised linear system. We use $10^4$ samples on each level $\ell$ to compute the estimates of rate variables $\alpha, \beta, \gamma$ in the MLMC complexity theorem (cf. Theorem \ref{thm:mlmc}).

\subsection{Test case I}

In the first experiment, we set $\bb{a}=[20;0]^T$ and use the Galerkin FEM to discretize the convection-diffusion equation. 
The stopping criteria for the Rayleigh quotient iteration and for the implicitly restarted Arnoldi method are set to be $10^{-12}$. In addition, for the implicitly restarted Arnoldi method, the Krylov subspace dimensions (the \texttt{ncv} values of ARPACK) are chosen empirically for each mesh size to optimize the number of Arnoldi iterations. They are $m=20, 40, 70, 70, 100$ for $h=2^{-3},2^{-4},2^{-5}, 2^{-6},2^{-7}$, respectively.

We demonstrate the efficiency of four variants of the MLMC method: ({\romannumeral 1}) the three-grid Rayleigh quotient iteration (tgRQI) with a model sequence defined by grid refinement; ({\romannumeral 2}) tgRQI with a model sequence defined by grid refinement and homotopy; ({\romannumeral 3}) the implicitly restarted Arnoldi method (IRAr) with a model sequence defined by grid refinement; and ({\romannumeral 4}) IRAr with a model sequence defined by grid refinement and homotopy.

{\bf ({\romannumeral 1}) MLMC with tgRQI:} Figure \ref{figc:rq1} illustrates the mean, the variance and the computational cost of multilevel differences $\lambda_\ell(\bsomega)-\lambda_{\ell-1}(\bsomega)$ of the smallest eigenvalue using tgRQI as the eigenvalue solver (without homotopy). Figure~\ref{figc:mean_rq1} also shows Monte Carlo estimates of the expected mean and variance of the smallest eigenvalue $\lambda_\ell(\bsomega)$ for each of the discretization levels. In addition to the computational cost, Figure \ref{figc:time_rq1} also shows the number of Rayleigh quotient iterations used at each level. We observe that the average number of iterations follows our analysis of the computational cost of tgRQI (cf. Alg.~\ref{alg:threegrid}). 
From these plots, we estimate that the rate variables in the MLMC complexity theorem are $\alpha\approx2.0$, $\beta\approx4.0$ and $\gamma\approx2.41$. 
Since the variance reduction rate $\beta$ is larger than the cost increase rate $\gamma$, the MLMC estimator is in the best case scenario, with $O(\varepsilon^{-2})$ complexity, as stated in Theorem \ref{thm:mlmc}.

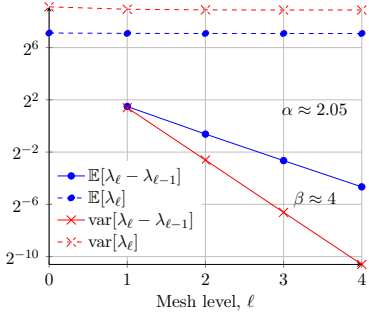
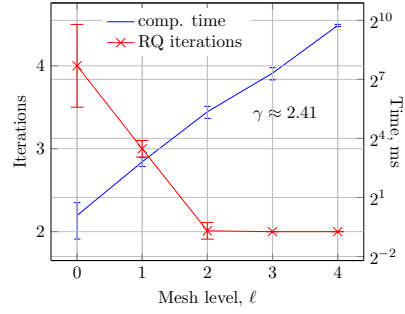
\begin{figure}[h!]
\centering
\begin{subfigure}{0.48\textwidth}
\vspace{0.5ex}
\begin{tikzpicture}[scale=0.6]
\begin{axis}[
    axis lines = left,
    grid = both,
	xlabel = {Mesh level, $\ell$},
    minor tick num=0,
    ylabel near ticks,
	ymode=log,
    log basis y={2},
    every axis y label/.style=
    {at={(-0.25,0.5)},rotate=90},
    y tick label style={
       /pgf/number format/.cd,
           sci,
           precision=100,
       /tikz/.cd,
    },
    xticklabel={%
        \pgfmathtruncatemacro{\IntegerTick}{\tick}%
        \pgfmathprintnumberto[verbatim,fixed,precision=1]{\tick}\tickAdjusted%
        \pgfmathparse{\IntegerTick == \tickAdjusted ? 1: 0}%
        \ifnum\pgfmathresult>0\relax$\IntegerTick$\else\fi%
    },
   legend style={at={(0.25,0.4)}, anchor=north,legend columns=1, draw=none},
   legend cell align={left},
]
\addplot[blue,mark=*,domain=1:1e-30]
	plot[error bars/.cd, y dir = both, y explicit]
	table[x index =0, y index=1]{results/v20/mean.txt};
	\addlegendentry{$\EE{E}[\lambda_\ell-\lambda_{\ell-1}]$}
\addplot[blue,dashed,mark=*,domain=1:1e-30]
	table[x index =0, y index=1]{results/v20/meanl.txt};
	\addlegendentry{$\EE{E}[\lambda_\ell]$}
\addplot[red,mark=x,mark size=4pt,domain=1:1e-30]
	table[x index =0, y index=1]{results/v20/var.txt};
	\addlegendentry{$\mathrm{var}[\lambda_\ell-\lambda_{\ell-1}]$}
\addplot[red,dashed,mark=x,mark size=4pt,domain=1:1e-30]
	table[x index =0, y index=1]{results/v20/varl.txt};
	\addlegendentry{$\mathrm{var}[\lambda_\ell]$}

\node at (rel axis cs:0.85,0.6) {$\alpha\approx 2.05$};
\node at (rel axis cs:0.85,0.25) { $\beta\approx 4$};
\end{axis}
\end{tikzpicture}
		\caption{Means and variances of $\lambda_\ell$ and $\lambda_\ell-\lambda_{\ell-1}$.}
		\label{figc:mean_rq1}
\end{subfigure}
\hfill
\begin{subfigure}{0.48\textwidth}
\begin{tikzpicture}[scale=0.6]
\begin{axis}[
   grid = both,
   xlabel = {Mesh level, $\ell$},
   ylabel = {Time, ms},
   every axis y label/.style={at={(1.1,0.5)},rotate=-90},
   axis y line*=right,
   minor tick num=0,
   ymode=log,
   log basis y={2},
     xticklabel={%
        \pgfmathtruncatemacro{\IntegerTick}{\tick}%
        \pgfmathprintnumberto[verbatim,fixed,precision=3]{\tick}\tickAdjusted%
        \pgfmathparse{\IntegerTick == \tickAdjusted ? 1: 0}%
        \ifnum\pgfmathresult>0\relax$\IntegerTick$\else\fi%
    },
   legend style={at={(0.3,0.65)}, anchor=north,legend columns=1, draw=none},
   legend cell align={left},
]
\addplot[blue,mark=*x
,domain=1:1e-30]
	plot[error bars/.cd, y dir = both, y explicit]
	table[x index =0, y index=1, y error index=2,]{results/v20/time_rq.txt};\label{plot_cpu_rq1}
\node at (rel axis cs:0.75,0.57) {$\gamma\approx 2.41$};
\end{axis}

\begin{axis}[
   grid=both,
   ylabel={Iterations},
   every axis y label/.style={at={(-.1,0.5)},rotate=90},
   axis y line*=left,
   xticklabels={},
   xtick={},   
   legend style={at={(0.4,0.98)}, anchor=north,legend columns=1, draw=none},
   legend cell align={left},
]
\addlegendimage{/pgfplots/refstyle=plot_cpu_rq1}\addlegendentry{comp. time}
\addplot[red,mark=x,mark size=4pt,domain=1:1e-30]
	plot[error bars/.cd, y dir = both, y explicit]
	table[x index =0, y index=1, y error index=2,]{results/v20/it_rq.txt};\addlegendentry{RQ iterations}
\end{axis}
\end{tikzpicture}
\caption{Computational time and average RQI.}
\label{figc:time_rq1}
\end{subfigure}

\caption{MLMC method using tgRQI for Test Case I with $\bb{a}=[20;0]^T$ and Galerkin FEM: (a) Mean (blue) and variance (red) of the eigenvalue $\lambda_\ell$ (dashed) and of $\lambda_\ell-\lambda_{\ell-1}$ (solid); (b) computational times for one multilevel difference (blue) and average number of Rayleigh quotient iterations (red) on each level. 
Where shown, the error bars represent $\pm$ one standard deviation.}
\label{figc:rq1}
\end{figure}

{\bf ({\romannumeral 2}) MLMC with homotopy and tgRQI:} Next, we consider the homotopy method in the MLMC setting together with tgRQI.
We use the conjecture in \eqref{eq:homotopy} to set the homotopy parameters such that $1 - t_\ell = O(h_\ell^2)$, $t_0 = 0$ and $t_L = 1$. For $L = 5$, this results in $t_\ell = \{0, 3/4, 15/16, 63/64, 1\}$. With this choice the eigenproblem on the zeroth level contains no convection term and is thus self-adjoint.
Figure~\ref{figc:h_mean} shows again the means and the variances of the multilevel differences $\lambda_\ell-\lambda_{\ell-1}$ in this setting, together with MC estimates of the expected means and variances of the eigenvalues for each level.
The hierarchy of homotopy parameters is chosen to guarantee good variance reduction for MLMC. Indeed, the variance of the multilevel difference decays smoothly with a rate $\beta\approx 3.65$. The expected mean of the difference, on the other hand, stagnates between $\ell = 1$ and $\ell = 2$. However, this initial stagnation is irrelevant for the MLMC complexity theorem; eventually for $\ell \ge 2$, the estimated means of the multilevel differences decrease again with a rate of $\alpha \approx 2$.
Figure \ref{figc:homo_it1} shows the number of Rayleigh quotient iterations used at each level and the computational cost, which grows with a rate of $\gamma\approx2.56$ here. 
This leads to the same asymptotic complexity for MLMC, since the regime is the same, i.e., $\beta>\gamma$, which is the optimal regime in Theorem \ref{thm:mlmc} with a complexity of $O(\varepsilon^{-2})$.

\begin{figure}[t!]
\centering
\begin{subfigure}{0.49\textwidth}
\begin{tikzpicture}[scale=0.6]
\begin{axis}[
   axis lines = left,
   grid = both,
	xlabel = {Mesh level, $\ell$},
   minor tick num=0,
   ylabel near ticks,
	ymode=log,
log basis y={2},
   every axis y label/.style=
{at={(-0.25,0.5)},rotate=90},
   y tick label style={
       /pgf/number format/.cd,
           sci,
           precision=100,
       /tikz/.cd,
   },
     xticklabel={%
        \pgfmathtruncatemacro{\IntegerTick}{\tick}%
        \pgfmathprintnumberto[verbatim,fixed,precision=1]{\tick}\tickAdjusted%
        \pgfmathparse{\IntegerTick == \tickAdjusted ? 1: 0}%
        \ifnum\pgfmathresult>0\relax$\IntegerTick$\else\fi%
    },
   legend style={at={(0.25,0.4)}, anchor=north,legend columns=1, draw=none},
   legend cell align={left},
]
\addplot[blue,mark=*,domain=1:1e-30]
	plot[error bars/.cd, y dir = both, y explicit]
	table[x index =0, y index=1]{results/v20/homo_mean.txt};
	\addlegendentry{$\EE{E}[\lambda_\ell-\lambda_{\ell-1}]$}
\node at (rel axis cs:0.85,0.42) {$\alpha\approx 2$};
\addplot[blue,dashed,mark=*,domain=1:1e-30]
	table[x index =0, y index=1]{results/v20/homo_meanl.txt};
	\addlegendentry{$\EE{E}[\lambda_\ell]$}
\addplot[red,mark=x,mark size=4pt,domain=1:1e-30]
	table[x index =0, y index=1]{results/v20/homo_var.txt};
	\addlegendentry{$\mathrm{var}[\lambda_\ell-\lambda_{\ell-1}]$}
\node at (rel axis cs:0.75,0.08) {$\beta\approx 3.65$};
\addplot[red,dashed,mark=x,mark size=4pt,domain=1:1e-30]
	table[x index =0, y index=1]{results/v20/homo_varl.txt};
	\addlegendentry{$\mathrm{var}[\lambda_\ell]$}
\end{axis}
\end{tikzpicture}
\caption{Means and variances of $\lambda_\ell$ and $\lambda_\ell-\lambda_{\ell-1}$.}
\label{figc:h_mean}
\end{subfigure}
\hfill
\begin{subfigure}{0.5\textwidth}
\begin{tikzpicture}[scale=0.6]
\begin{axis}[
   grid = both,
   xlabel = {Mesh level, $\ell$},
   ylabel = {Time, ms},
   every axis y label/.style={at={(1.1,0.5)},rotate=-90},
   axis y line*=right,
   minor tick num=0,
   ymode=log,
   log basis y={2},
     xticklabel={%
        \pgfmathtruncatemacro{\IntegerTick}{\tick}%
        \pgfmathprintnumberto[verbatim,fixed,precision=3]{\tick}\tickAdjusted%
        \pgfmathparse{\IntegerTick == \tickAdjusted ? 1: 0}%
        \ifnum\pgfmathresult>0\relax$\IntegerTick$\else\fi%
    },
   legend style={at={(0.3,0.65)}, anchor=north,legend columns=1, draw=none},
   legend cell align={left},
]
\addplot[blue,mark=*x
,domain=1:1e-30]
	plot[error bars/.cd, y dir = both, y explicit]
	table[x index =0, y index=1, y error index=2,]{results/v20/homo_time.txt};\label{figc:homo_time1}
\end{axis}

\begin{axis}[
   grid=both,
   ylabel={Iterations},
   every axis y label/.style={at={(-.1,0.5)},rotate=90},
   axis y line*=left,
   xticklabels={},
   xtick={},   
   ytick={2,3,4},
   legend style={at={(0.4,0.98)}, anchor=north,legend columns=1, draw=none},
   legend cell align={left},
]
\addlegendimage{/pgfplots/refstyle=plot_cpu}\addlegendentry{comp. time}
\addplot[red,mark=x,mark size=4pt,domain=1:1e-30]
	plot[error bars/.cd, y dir = both, y explicit]
	table[x index =0, y index=1, y error index=2,]{results/v20/homo_it.txt};\addlegendentry{RQ iterations}
\node at (rel axis cs:0.62,0.7) {$\gamma\approx 2.56$};
\end{axis}
\end{tikzpicture}
		\caption{Computational times and average RQI.}
		\label{figc:homo_it1}
\end{subfigure}
\caption{MLMC method using homotopy and tgRQI for Test Case I with $\bb{a}=[20;0]^T$ and Galerkin FEM: (a) Mean (blue) and variance (red) of the eigenvalue $\lambda_\ell$ (dashed) and of $\lambda_\ell-\lambda_{\ell-1}$ (solid); (b) computational times for one multilevel difference (blue) and average number of RQIs (red) on each level. Where shown, the error bars represent $\pm$ one standard deviation.}
\label{figc:homo_mlmc1}
\end{figure}

{\bf ({\romannumeral 3}) MLMC with IRAr:} Similar results are obtained by using the implicitly restarted Arnoldi eigenvalue solver (without homotopy). Since the mean and the variance of the multilevel differences in this setting are almost identical to those of the Rayleigh quotient solver, we omit the plots here and only report the computational cost. Figure~\ref{figc:ARPACK_nohomotop} shows the average number of matrix-vector products and the estimated CPU time for computing each of the multilevel differences, which grows with a rate of $\gamma \approx 3.5$. Here, the increasing dimension of Krylov subspaces with grid refinement likely causes the higher growth rate of computational time compared to the experiment using tgRQI. Nonetheless, the MLMC estimator has again the optimal $O(\varepsilon^{-2})$ complexity.

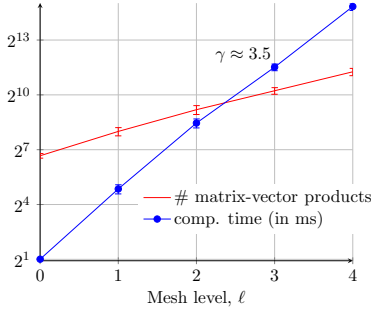
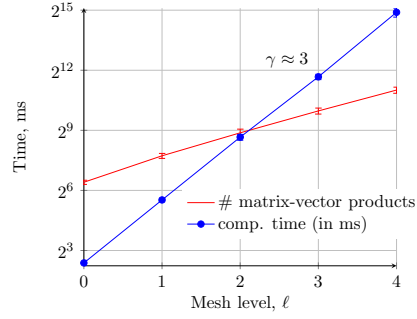
\begin{figure}[t!]
\centering
\begin{subfigure}{0.48\textwidth}
\centering
\begin{tikzpicture}[scale=0.6]
\begin{axis}[
    axis lines = left,
    grid = both,
    xlabel = {Mesh level, $\ell$},
    minor tick num=0,
    ylabel near ticks,
	ymode=log,
    log basis y={2},
    every axis y label/.style=
    {at={(-0.2,0.5)},rotate=90},
        y tick label style={
        /pgf/number format/.cd,
        fixed,
        sci,
        precision=7,
       /tikz/.cd,
     },
     xticklabel={%
        \pgfmathtruncatemacro{\IntegerTick}{\tick}%
        \pgfmathprintnumberto[verbatim,fixed,precision=3]{\tick}\tickAdjusted%
        \pgfmathparse{\IntegerTick == \tickAdjusted ? 1: 0}%
        \ifnum\pgfmathresult>0\relax$\IntegerTick$\else\fi%
    },
   legend style={at={(0.7,0.3)}, anchor=north,legend columns=1, draw=none},
   legend cell align={left},
]
\addplot[red,mark=*x
,domain=1:1e-30]
	plot[error bars/.cd, y dir = both, y explicit]
	table[x index =0, y index=1, y error index=2,]{results/v20/Ax.txt};\addlegendentry{\# matrix-vector products}
 
\addplot[blue,mark=*,domain=1:1e-30]
	plot[error bars/.cd, y dir = both, y explicit]
	table[x index =0, y index=1, y error index=2,]{results/v20/time.txt}; \addlegendentry{comp. time (in ms)}
\node at (rel axis cs:0.65,0.8) {$\gamma\approx 3.5$};
\end{axis}
\end{tikzpicture}
\caption{Without homotopy.}
		\label{figc:ARPACK_nohomotop}
\end{subfigure}
\hfill
\begin{subfigure}{0.48\textwidth}
\begin{tikzpicture}[scale=0.6]
\begin{axis}[
   axis lines = left,
   grid = both,
	xlabel = {Mesh level, $\ell$},
	ylabel = {Time, ms},
   minor tick num=0,
   ylabel near ticks,
	ymode=log,
log basis y={2},
   every axis y label/.style=
{at={(-0.2,0.5)},rotate=90},
   y tick label style={
       /pgf/number format/.cd,
          fixed,
           sci,
           precision=7,
       /tikz/.cd,
   },
     xticklabel={%
        \pgfmathtruncatemacro{\IntegerTick}{\tick}%
        \pgfmathprintnumberto[verbatim,fixed,precision=3]{\tick}\tickAdjusted%
        \pgfmathparse{\IntegerTick == \tickAdjusted ? 1: 0}%
        \ifnum\pgfmathresult>0\relax$\IntegerTick$\else\fi%
    },
   legend style={at={(0.7,0.3)}, anchor=north,legend columns=1, draw=none},
   legend cell align={left},
]

\addplot[red,mark=*x
,domain=1:1e-30]
	plot[error bars/.cd, y dir = both, y explicit]
	table[x index =0, y index=1, y error index=2,]{results/v20/Ax_homo.txt};\addlegendentry{\# matrix-vector products}
 
\addplot[blue,mark=*,domain=1:1e-30]
	plot[error bars/.cd, y dir = both, y explicit]
	table[x index =0, y index=1, y error index=2,]{results/v20/time_homo.txt};\addlegendentry{comp. time (in ms)}
\node at (rel axis cs:0.65,0.8) {$\gamma\approx 3$};
\end{axis}
\end{tikzpicture}
\caption{With homotopy.}
		\label{figc:ARPACK_homotop}
\end{subfigure}
\caption{MLMC method using IRAr for Test Case I with $\bb{a}=[20;0]^T$ and Galerkin FEM, both without (a) and with (b) homotopy: average computational cost (blue) and average number of matrix-vector products (red) per sample of $\lambda_\ell-\lambda_{\ell-1}$. The error bars represent $\pm$ one standard deviation.}
\label{figc:ARPACK}
\end{figure}

{\bf ({\romannumeral 4}) MLMC with homotopy and IRAr:} Finally, we consider the behaviour of IRAr
with homotopy, using the same sequence for the homotopy parameter $t_\ell$ as in ({\romannumeral 2}). Again, 
we only focus on computational cost, 
showing the average number of matrix-vector products and the CPU time for computing each of the multilevel differences in Figure \ref{figc:ARPACK_homotop}. As in ({\romannumeral 2}), the cost grows at a rate of $\gamma \approx 3$ leading again to the optimal $O(\varepsilon^{-2})$ complexity for MLMC.

\begin{figure}[t!]
\centering
\begin{tikzpicture}[scale=0.65]
\begin{axis}[
   axis lines = left,
   grid = both,
   xlabel = root mean square error,
	ylabel = {CPU time},
   minor tick num=0,
   ylabel near ticks,
	ymode=log,
	xmode=log,
log basis y={2},
	log basis x={2},
   every axis y label/.style=
{at={(-0.2,0.5)},rotate=90},
   y tick label style={
       /pgf/number format/.cd,
           sci,
           precision=100,
       /tikz/.cd,
   },
   legend style={at={(1.1,1.25)}, anchor=north,legend columns=1, draw=none},
   legend cell align={left},
]
\addplot plot[mark=*,blue,domain=1:1e-30,dashed]
        file {results/v20/mse_ar_mc.txt};
	\addlegendentry{Standard Monte Carlo with Arnoldi method}
\addplot plot[mark=*,red,domain=1:1e-30,dashed]
        file {results/v20/mse_rq_mc.txt};
\addlegendentry{Standard Monte Carlo with Rayleigh quotient}
\addplot plot[mark=square*,blue,domain=1:1e-30]
        file {results/v20/mse_homo_ar.txt};
	\addlegendentry{MLMC with Arnoldi and homotopy}
\addplot plot[mark=square*,red,domain=1:1e-30]
        file {results/v20/mse_homo_rq.txt};
	\addlegendentry{MLMC with Rayleigh quotient and homotopy}
\addplot plot[mark=triangle*,red,domain=1:1e-30]
        file {results/v20/mse_rq.txt};
	\addlegendentry{MLMC with Rayleigh quotient but without homotopy}
\addplot plot[mark=triangle*,blue,domain=1:1e-30]
        file {results/v20/mse_ar.txt};
	\addlegendentry{MLMC with Arnoldi but without homotopy}

\end{axis}
\end{tikzpicture}
\caption{CPU time vs. root mean square error of all estimators in Test Case I.}
\label{fig:cpu_mse_v20}
\end{figure}
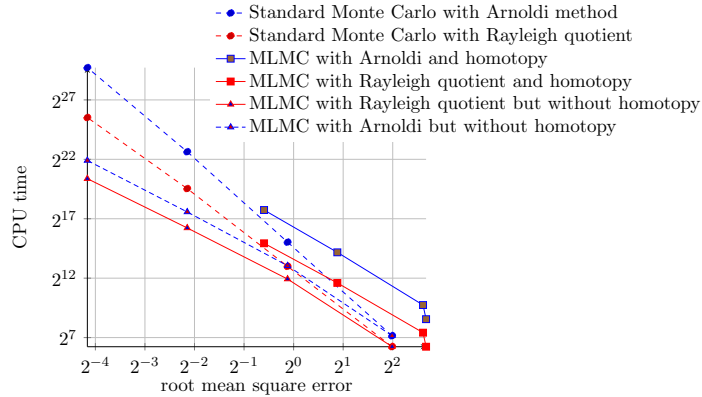

{\bf Overall comparison:} In Figure~\ref{fig:cpu_mse_v20}, we show the CPU time versus the root mean square error for all four presented MLMC estimators together, as well as for standard Monte Carlo estimators using tgRQI (red) and IRAr (blue). The estimated complexity of standard Monte Carlo methods are $O(\varepsilon^{-2.92})$ and $O(\varepsilon^{-3.35})$ for tgRQI and IRAr, respectively. Overall, MLMC using tgRQI (without homotopy) outperforms all other methods, despite that all four MLMC methods achieve the optimal $O(\varepsilon^{-2})$ complexity.

\subsection{Test case II}
\label{sec:problem2}
For the second experiment, we increase the velocity to $\bb{a}=[50;0]^T$ and focus on the comparison between Galerkin and SUPG discretizations. Thus, we only consider the three-grid Rayleigh quotient iteration (tgRQI) with a multilevel sequence based on geometrically refined grids without homotopy. Note that for such a strong convection, five steps in the homotopy approach are insufficient: the eigenvalues for consecutive homotopy parameters are too different to achieve variance reduction in the homotopy-based MLMC method. Its computational complexity is almost the same as the complexity of standard Monte Carlo, namely almost $O(\varepsilon^{-3.5})$. The performance of MLMC with implicitly restarted Arnoldi on the other hand is similar to MLMC with tgRQI.

\begin{figure}[t!]
\centering
\begin{subfigure}{0.49\textwidth}
\centering
\vspace{1.4mm}
\begin{tikzpicture}[scale=0.6]
\begin{axis}[
   axis lines = left,
   grid = both,
	xlabel = {Mesh level, $\ell$},
   minor tick num=0,
   ylabel near ticks,
	ymode=log,
log basis y={2},
   every axis y label/.style=
{at={(-0.25,0.5)},rotate=90},
   y tick label style={
       /pgf/number format/.cd,
           sci,
           precision=100,
       /tikz/.cd,
   },
     xticklabel={%
        \pgfmathtruncatemacro{\IntegerTick}{\tick}%
        \pgfmathprintnumberto[verbatim,fixed,precision=1]{\tick}\tickAdjusted%
        \pgfmathparse{\IntegerTick == \tickAdjusted ? 1: 0}%
        \ifnum\pgfmathresult>0\relax$\IntegerTick$\else\fi%
    },
   legend style={at={(0.25,0.4)}, anchor=north,legend columns=1, draw=none},
   legend cell align={left},
]
\addplot[blue,mark=*,domain=1:1e-30]
	plot[error bars/.cd, y dir = both, y explicit]
	table[x index =0, y index=1]{results/v50/rq_mean.txt};
	\addlegendentry{$\EE{E}[\lambda_\ell-\lambda_{\ell-1}]$}
\addplot[blue,dashed,mark=*,domain=1:1e-30]
	table[x index =0, y index=1]{results/v50/rq_meanl.txt};
	\addlegendentry{$\EE{E}[\lambda_\ell]$}
\addplot[red,mark=x,mark size=4pt,domain=1:1e-30]
	table[x index =0, y index=1]{results/v50/rq_var.txt};
	\addlegendentry{$\mathrm{var}[\lambda_\ell-\lambda_{\ell-1}]$}
\addplot[red,dashed,mark=x,mark size=4pt,domain=1:1e-30]
	table[x index =0, y index=1]{results/v50/rq_varl.txt};
	\addlegendentry{$\mathrm{var}[\lambda_\ell]$}
\end{axis}
\end{tikzpicture}
\caption{Means and variances of $\lambda_\ell$ and $\lambda_\ell-\lambda_{\ell-1}$.}
\label{figc:rq_mean2}
\end{subfigure}
\hfill
\begin{subfigure}{0.49\textwidth}
\centering
\begin{tikzpicture}[scale=0.6]
\begin{axis}[
   grid = both,
   xlabel = {Mesh level, $\ell$},
   ylabel = {Time, ms},
   every axis y label/.style={at={(1.1,0.5)},rotate=-90},
   axis y line*=right,
   minor tick num=0,
   ymode=log,
   log basis y={2},
     xticklabel={%
        \pgfmathtruncatemacro{\IntegerTick}{\tick}%
        \pgfmathprintnumberto[verbatim,fixed,precision=3]{\tick}\tickAdjusted%
        \pgfmathparse{\IntegerTick == \tickAdjusted ? 1: 0}%
        \ifnum\pgfmathresult>0\relax$\IntegerTick$\else\fi%
    },
   legend style={at={(0.3,0.7)}, anchor=north,legend columns=1, draw=none},
   legend cell align={left},
]
\addplot[blue,mark=*x
,domain=1:1e-30]
	plot[error bars/.cd, y dir = both, y explicit]
	table[x index =0, y index=1, y error index=2,]{results/v50/rq_time.txt};\label{plot_cpu}
 \node at (rel axis cs:0.62,0.7) {$\gamma\approx 1.88$};
\end{axis}

\begin{axis}[
   grid=both,
   ylabel={Iterations},
   every axis y label/.style={at={(-.1,0.5)},rotate=90},
   axis y line*=left,
   xticklabels={},
   xtick={},   
   legend style={at={(0.4,0.98)}, anchor=north,legend columns=1, draw=none},
   legend cell align={left},
]
\addlegendimage{/pgfplots/refstyle=plot_cpu}\addlegendentry{comp. time}
\addplot[red,mark=x,mark size=4pt,domain=1:1e-30]
	plot[error bars/.cd, y dir = both, y explicit]
	table[x index =0, y index=1, y error index=2,]{results/v50/rq_it.txt};\addlegendentry{RQ iterations}
\end{axis}
\end{tikzpicture}
		\caption{Computational times and average RQI.}
		\label{figc:time_rq2}
\end{subfigure}
\caption{MLMC method using tgRQI for Test Case II with $\bb{a}=[50;0]^T$ and Galerkin FEM: (a) Mean (blue) and variance (red) of the eigenvalue $\lambda_\ell$ (dashed) and of $\lambda_\ell-\lambda_{\ell-1}$ (solid); (b) computational time for one multilevel difference (blue) and average number of Rayleigh quotient iterations (red) on each level. Where shown, the error bars represent $\pm$ one standard deviation.}
\label{figc:rq_mlmc2}
\end{figure}
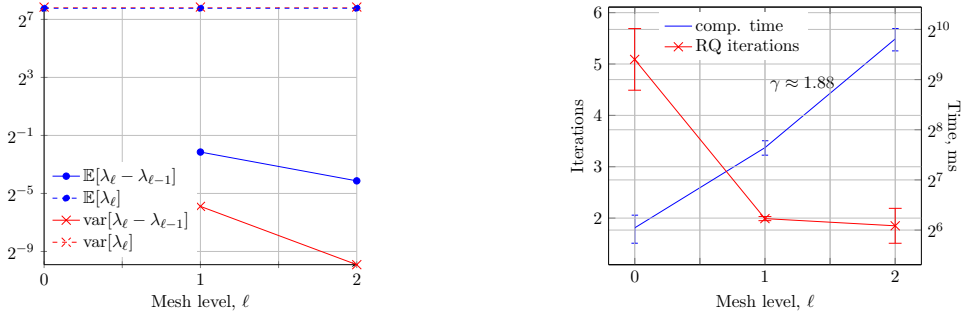

{\bf Galerkin:} Due to the higher convection velocity the first two levels are unstable for most of the realizations of $\bsomega$ as the FEM solution may exhibit non-physical oscillations. Thus, we set the coarsest level for the MLMC method to $h_0=2^{-5}$ here. Keeping the same finest grid level $h_L=2^{-7}$, this means that we only use a total of three levels ($L=2$) compared to the sequence in Test Case~I, which had a total of five levels ($L=4$).
Figure \ref{figc:rq_mean2} shows the expectation and variance of the multilevel differences. Here, we only have a couple of data points for estimating the rate variables of the MLMC complexity theorem, but the estimates are $\alpha \approx 2$ and $\beta \approx 4$ as expected theoretically. 
The average number of Rayleigh quotient iterations in
Figure \ref{figc:time_rq2} also behaves as in Test Case I with $5$ iterations on the coarsest level and $2$ iterations on the subsequent levels
as expected for the three-grid Rayleigh quotient iteration (Alg.~\ref{alg:threegrid}) -- recall that Levels 1 and 2 here correspond to Levels 3 and 4 in Figures \ref{figc:time_rq1} and \ref{figc:homo_it1}. The estimated value for $\gamma \approx 1.88$, and thus the MLMC complexity is still $O(\varepsilon^{-2})$. However, we cannot use as many levels due the numerical stability issues caused by the higher convection velocity, which substantially increases the prefactor in the $O(\varepsilon^{-2})$ cost of the algorithm. 

{\bf SUPG:} By using the SUPG discretization, we overcome the numerical stability issue and can use all five levels in MLMC, starting with $h_0 = 2^{-3}$. As can be seen in Figure~\ref{figc:supg_mean}, the expectation and the variance of the multilevel differences converge with the same rates as for the Galerkin FEM, namely $\alpha \approx 2$ and $\beta \approx 4$ respectively. Also, clearly the use of SUPG leads to stable estimates even on the coarser levels. Figure~\ref{figc:supg_time} reports the average number of Rayleigh quotient iterations used at each level and the computational cost. We estimate that the computational cost increases at a rate of $\gamma\approx2.33$ here. In any case, the use of SUPG in the MLMC  also results in the optimal $O(\varepsilon^{-2})$ complexity. 

\begin{figure}[t!]
\centering
\begin{subfigure}{0.49\textwidth}
\begin{tikzpicture}[scale=0.6]
\begin{axis}[
   axis lines = left,
   grid = both,
	xlabel = {Mesh level, $\ell$},
   minor tick num=0,
   ylabel near ticks,
	ymode=log,
log basis y={2},
   every axis y label/.style=
{at={(-0.25,0.5)},rotate=90},
   y tick label style={
       /pgf/number format/.cd,
           sci,
           precision=100,
       /tikz/.cd,
   },
     xticklabel={%
        \pgfmathtruncatemacro{\IntegerTick}{\tick}%
        \pgfmathprintnumberto[verbatim,fixed,precision=1]{\tick}\tickAdjusted%
        \pgfmathparse{\IntegerTick == \tickAdjusted ? 1: 0}%
        \ifnum\pgfmathresult>0\relax$\IntegerTick$\else\fi%
    },
   legend style={at={(0.25,0.4)}, anchor=north,legend columns=1, draw=none},
   legend cell align={left},
]
\addplot[blue,mark=*,domain=1:1e-30]
	plot[error bars/.cd, y dir = both, y explicit]
	table[x index =0, y index=1]{results/v50/supg_mean.txt};
	\addlegendentry{$\EE{E}[\lambda_\ell-\lambda_{\ell-1}]$}
\node at (rel axis cs:0.65,0.55) {$\alpha\approx 2$};
\addplot[blue,dashed,mark=*,domain=1:1e-30]
	table[x index =0, y index=1]{results/v50/supg_meanl.txt};
	\addlegendentry{$\EE{E}[\lambda_\ell]$}
\addplot[red,mark=x,mark size=4pt,domain=1:1e-30]
	table[x index =0, y index=1]{results/v50/supg_var.txt};
	\addlegendentry{$\mathrm{var}[\lambda_\ell-\lambda_{\ell-1}]$}
\node at (rel axis cs:0.65,0.2) {$\beta\approx 4$};
\addplot[red,dashed,mark=x,mark size=4pt,domain=1:1e-30]
	table[x index =0, y index=1]{results/v50/supg_varl.txt};
	\addlegendentry{$\mathrm{var}[\lambda_\ell]$}
\end{axis}
\end{tikzpicture}
		\caption{Means and variances of $\lambda_\ell$ and $\lambda_\ell-\lambda_{\ell-1}$.}
		\label{figc:supg_mean}
\end{subfigure}
\hfill
\begin{subfigure}{0.49\textwidth}
\begin{tikzpicture}[scale=0.6]
\begin{axis}[
   grid = both,
   xlabel = {Mesh level, $\ell$},
   ylabel = {Time, ms},
   every axis y label/.style={at={(1.1,0.5)},rotate=-90},
   axis y line*=right,
   minor tick num=0,
   ymode=log,
   log basis y={2},
     xticklabel={%
        \pgfmathtruncatemacro{\IntegerTick}{\tick}%
        \pgfmathprintnumberto[verbatim,fixed,precision=3]{\tick}\tickAdjusted%
        \pgfmathparse{\IntegerTick == \tickAdjusted ? 1: 0}%
        \ifnum\pgfmathresult>0\relax$\IntegerTick$\else\fi%
    },
   legend style={at={(0.3,0.7)}, anchor=north,legend columns=1, draw=none},
   legend cell align={left},
]
\addplot[blue,mark=*x
,domain=1:1e-30]
	plot[error bars/.cd, y dir = both, y explicit]
	table[x index =0, y index=1, y error index=2,]{results/v50/supg_time.txt};\label{figc:supg_time1}
\end{axis}

\begin{axis}[
   grid=both,
   ylabel={Iterations},
   every axis y label/.style={at={(-.1,0.5)},rotate=90},
   axis y line*=left,
   xticklabels={},
   xtick={},   
   legend style={at={(0.4,0.98)}, anchor=north,legend columns=1, draw=none},
   legend cell align={left},
]
\addlegendimage{/pgfplots/refstyle=plot_cpu}\addlegendentry{comp. time}
\addplot[red,mark=x,mark size=4pt,domain=1:1e-30]
	plot[error bars/.cd, y dir = both, y explicit]
	table[x index =0, y index=1, y error index=2,]{results/v50/supg_it.txt};\addlegendentry{RQ iterations}
\node at (rel axis cs:0.65,0.7) {$\gamma\approx 2.23$};
\end{axis}
\end{tikzpicture}
		\caption{Computational times and average RQI.}
		\label{figc:supg_time}
\end{subfigure}
\caption{MLMC method using tgRQI for Test Case II with $\bb{a}=[50;0]^T$ and SUPG discretization: (a) Mean (blue) and variance (red) of the eigenvalue $\lambda_\ell$ (dashed) and of $\lambda_\ell-\lambda_{\ell-1}$ (solid); (b) computational time for one multilevel difference (blue) and average number of Rayleigh quotient iterations (red) on each level. Where shown, the error bars represent $\pm$ one standard deviation.}
\label{figc:supg_mlmc}
\end{figure}
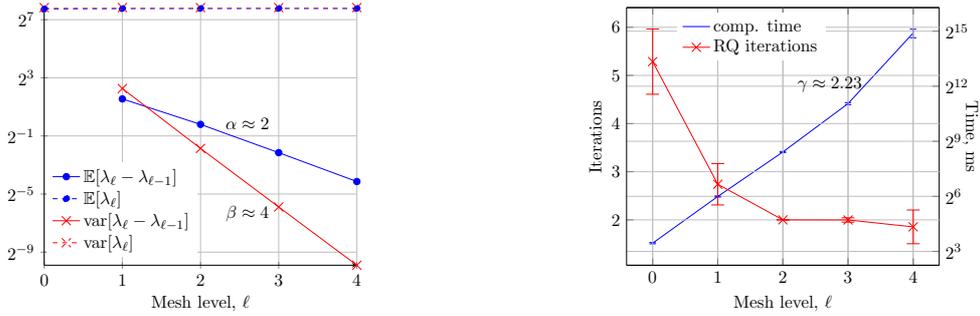

\begin{figure}[t!]
\centering
\begin{tikzpicture}[scale=0.65]
\begin{axis}[
   axis lines = left,
   grid = both,
   xlabel = root mean square error,
	ylabel = {CPU time},
   minor tick num=0,
	ymode=log,
	xmode=log,
log basis y={2},
	log basis x={2},
   every axis y label/.style=
{at={(-0.2,0.5)},rotate=90},
   y tick label style={
       /pgf/number format/.cd,
           fixed,
           sci,
           precision=7,
       /tikz/.cd,
   },
   legend style={at={(1.1,1.05)}, anchor=north,legend columns=1, draw=none},
   legend cell align={left},
]
\addplot plot[mark=*,red,domain=1:1e-30]
        file {results/v50/mse_mc.txt};
	\addlegendentry{Standard Monte Carlo with Rayleigh quotient and Galerkin}
\addplot plot[mark=triangle*,red,domain=1:1e-30]
        file {results/v50/mse.txt};
	\addlegendentry{MLMC with Rayleigh quotient and Galerkin}
\addplot plot[mark=*,black,domain=1:1e-30]
        file {results/v50/mse_supg.txt};
	\addlegendentry{MLMC with Rayleigh quotient and SUPG}

\end{axis}
\end{tikzpicture}
 \caption{CPU time vs.~root mean square error of the estimators in Test Case II.}
\label{fig:cpu_supg}
\end{figure}
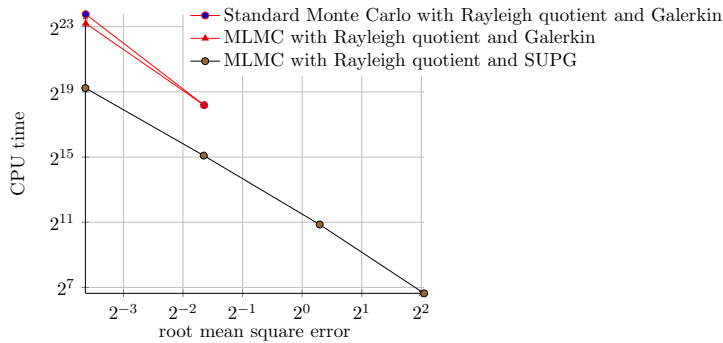

{\bf Overall comparison:} Figure~\ref{fig:cpu_supg} shows CPU times versus root mean square errors for the MLMC methods (with tgRQI and without homotopy) using Galerkin FEM and SUPG discretizations. They are compared to a standard Monte Carlo method with Galerkin FEM. Although both MLMC estimates have the optimal  $O(\varepsilon^{-2})$ complexity, the stability offered by SUPG enables us to use more, coarser levels, thus leading to a smaller prefactor and a significant computational gain of a factor 10-20 over the Galerkin FEM based method.  

\subsection{Test Case III}
In this experiment, the convection velocity becomes a divergence-free random field generated using \eqref{eq:random_vel1} and \eqref{eq:random_vel2}. We discretise the eigenvalue problem using SUPG and apply the three-grid Rayleigh quotient iteration (tgRQI) without homotopy to solve multilevel eigenvalue problems. The stopping criteria for tgRQI is set to be $10^{-12}$. The same sequence of grid refinements, $h=2^{-3},2^{-4},2^{-5}, 2^{-6},2^{-7}$, as in previous test cases is used to construct multilevel estimators.

{\bf MLMC:} Figure \ref{figc:v_mlmc3} illustrates the mean, the variance and the computational cost of multilevel differences $\lambda_\ell(\bsomega)-\lambda_{\ell-1}(\bsomega)$ of the smallest eigenvalue using tgRQI as the eigenvalue solver. Figure~\ref{figc:h_mean_v} also shows Monte Carlo estimates of the expected mean and variance of the smallest eigenvalue $\lambda_\ell(\bsomega)$ for each of the discretization levels. In addition to the computational cost, Figure \ref{figc:v_it1} also shows the number of Rayleigh quotient iterations used at each level. We observe that the average number of iterations follows our analysis of the computational cost of tgRQI (cf. Alg.~\ref{alg:threegrid}). 
From these plots, we estimate that the rate variables in the MLMC complexity theorem are $\alpha\approx2.0$, $\beta\approx4$ and $\gamma\approx2.23$. Since the variance reduction rate $\beta$ is larger than the cost increase rate $\gamma$, the MLMC estimator is in the best case scenario, with $O(\varepsilon^{-2})$ complexity, as stated in Theorem \ref{thm:mlmc}. In Figure~\ref{fig:cpu_mse_v_random}, we compare the computational complexity of MLMC to that of the standard Monte Carlo. Numerically, we observe that the CPU time of MLMC is approximately $O(\varepsilon^{-2.06})$, which is close to the theoretically predicted rate. In comparison, the CPU time of the standard MC is approximately $O(\varepsilon^{-3.2})$ in this test case.

{\bf MLQMC:} All QMC computations were implemented using Dirk Nuyens' code accompanying
\cite{KN16} and use a randomly shifted embedded lattice rule in base 2, as outlined in \cite{CKN06},
with $32$ i.i.d.~random shifts.
In Figure~\ref{fig:qmc_error}, we plot convergence of the MSE for both MC and QMC for three different cases: for $\lambda_0$ in plot (a),
for the difference $\lambda_1 - \lambda_0$ in plot (b), 
and for the difference $\lambda_2 - \lambda_1$ in plot (c).
Here the meshwidths are given by $h_0 = 2^{-3}$, $h_1 = 2^{-4}$ and $h_2 = 2^{-5}$.
In all cases, QMC outperforms MC, where for $\lambda_0$ the MSE for QMC converges at an observed rate
of $-1.78$, whereas MC converges with the rate $-1$.
For the other two cases, which are MSEs of multilevel differences, the QMC converges with an approximate rate of $-1.63$, which is again clearly faster than the MC convergence rate of $-1$.
This observed MSE convergence for the QMC approximations of the differences implies that \emph{II(b)} holds with $\eta \approx 0.61$.
For MLQMC, to choose $N_\ell$ we use \eqref{eq:N_ell_mlqmc} with $\eta \approx 0.61$ and with 
$N_0$ scaled such that the overall MSE is less than $\varepsilon^2/\sqrt{2}$ for each tolerance $\varepsilon$.
Since we use a base-2 lattice rule, we round up $N_\ell$ to the next power of 2.

The MLQMC complexity, in terms of CPU time, is plotted in Figure~\ref{fig:cpu_mse_v_random}, along with the 
results for MC and MLMC. Comparing the three methods in Figure~\ref{fig:cpu_mse_v_random}, clearly MLQMC provides the best complexity, followed by MLMC then standard MC.
In this case, we have the approximate rates $\beta \eta \approx 4 \times 0.61 = 2.44 > \gamma \approx 2.23 $, which implies that for MLQMC we are in the optimal regime for the cost with 
$C_\mathrm{MLQMC}(\varepsilon) \lesssim \varepsilon^{-2\eta}$.
Numerically, we observe that the rate is given by $1.28$, which is very close to the theoretically predicted rate of $2\eta \approx 1.22$.

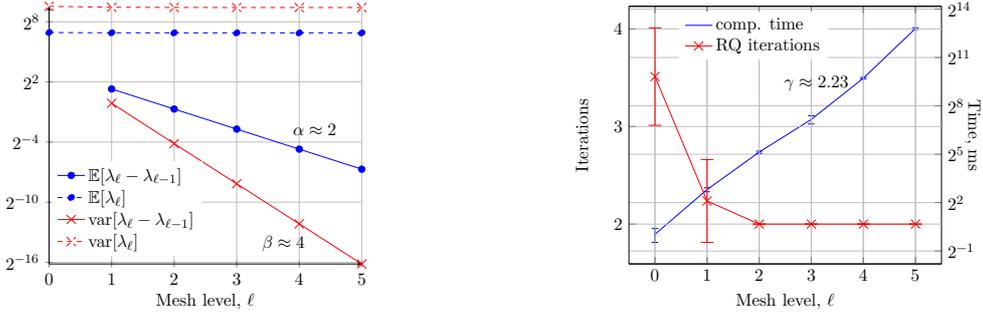
\begin{figure}[t!]
\centering
\begin{subfigure}{0.49\textwidth}
\begin{tikzpicture}[scale=0.6]
\begin{axis}[
   axis lines = left,
   grid = both,
	xlabel = {Mesh level, $\ell$},
   minor tick num=0,
   ylabel near ticks,
	ymode=log,
log basis y={2},
   every axis y label/.style=
{at={(-0.25,0.5)},rotate=90},
   y tick label style={
       /pgf/number format/.cd,
           sci,
           precision=100,
       /tikz/.cd,
   },
     xticklabel={%
        \pgfmathtruncatemacro{\IntegerTick}{\tick}%
        \pgfmathprintnumberto[verbatim,fixed,precision=1]{\tick}\tickAdjusted%
        \pgfmathparse{\IntegerTick == \tickAdjusted ? 1: 0}%
        \ifnum\pgfmathresult>0\relax$\IntegerTick$\else\fi%
    },
   legend style={at={(0.25,0.4)}, anchor=north,legend columns=1, draw=none},
   legend cell align={left},
]
\addplot[blue,mark=*,domain=1:1e-30]
	plot[error bars/.cd, y dir = both, y explicit]
	table[x index =0, y index=1]{results/eg3/meanl.txt};
	\addlegendentry{$\EE{E}[\lambda_\ell-\lambda_{\ell-1}]$}
\node at (rel axis cs:0.85,0.52) {$\alpha\approx 2$};
\addplot[blue,dashed,mark=*,domain=1:1e-30]
	table[x index =0, y index=1]{results/eg3/mean.txt};
	\addlegendentry{$\EE{E}[\lambda_\ell]$}
\addplot[red,mark=x,mark size=4pt,domain=1:1e-30]
	table[x index =0, y index=1]{results/eg3/varl.txt};
	\addlegendentry{$\mathrm{var}[\lambda_\ell-\lambda_{\ell-1}]$}
\node at (rel axis cs:0.75,0.08) {$\beta\approx 4$};
\addplot[red,dashed,mark=x,mark size=4pt,domain=1:1e-30]
	table[x index =0, y index=1]{results/eg3/var.txt};
	\addlegendentry{$\mathrm{var}[\lambda_\ell]$}
\end{axis}
\end{tikzpicture}
\caption{Means and variances of $\lambda_\ell$ and $\lambda_\ell-\lambda_{\ell-1}$.}
\label{figc:h_mean_v}
\end{subfigure}
\hfill
\begin{subfigure}{0.5\textwidth}
\begin{tikzpicture}[scale=0.6]
\begin{axis}[
   grid = both,
   xlabel = {Mesh level, $\ell$},
   ylabel = {Time, ms},
   every axis y label/.style={at={(1.1,0.5)},rotate=-90},
   axis y line*=right,
   minor tick num=0,
   ymode=log,
   log basis y={2},
     xticklabel={%
        \pgfmathtruncatemacro{\IntegerTick}{\tick}%
        \pgfmathprintnumberto[verbatim,fixed,precision=3]{\tick}\tickAdjusted%
        \pgfmathparse{\IntegerTick == \tickAdjusted ? 1: 0}%
        \ifnum\pgfmathresult>0\relax$\IntegerTick$\else\fi%
    },
   legend style={at={(0.3,0.7)}, anchor=north,legend columns=1, draw=none},
   legend cell align={left},
]
\addplot[blue,mark=*x
,domain=1:1e-30]
	plot[error bars/.cd, y dir = both, y explicit]
	table[x index =0, y index=1, y error index=2,]{results/eg3/time.txt};\label{figc:v_time1}
\end{axis}

\begin{axis}[
   grid=both,
   ylabel={Iterations},
   every axis y label/.style={at={(-.15,0.5)},rotate=90},
   axis y line*=left,
   xticklabels={},
   xtick={},   
   ytick={2,3,4},
   legend style={at={(0.4,0.98)}, anchor=north,legend columns=1, draw=none},
   legend cell align={left},
]
\addlegendimage{/pgfplots/refstyle=plot_cpu}\addlegendentry{comp. time}
\addplot[red,mark=x,mark size=4pt,domain=1:1e-30]
	plot[error bars/.cd, y dir = both, y explicit]
	table[x index =0, y index=1, y error index=2,]{results/eg3/it.txt};\addlegendentry{RQ iterations}
\node at (rel axis cs:0.6,0.7) {$\gamma\approx 2.23$};
2/1.63-1\end{axis}
\end{tikzpicture}
		\caption{Computational times and average RQI.}
		\label{figc:v_it1}
\end{subfigure}
\caption{MLMC method using tgRQI and SUPG for Test Case III with random velocity and random conductivity: (a) Mean (blue) and variance (red) of the eigenvalue $\lambda_\ell$ (dashed) and of $\lambda_\ell-\lambda_{\ell-1}$ (solid); (b) computational times for one multilevel difference (blue) and average number of RQIs (red) on each level. Where shown, the error bars represent $\pm$ one standard deviation.}
\label{figc:v_mlmc3}
\end{figure}

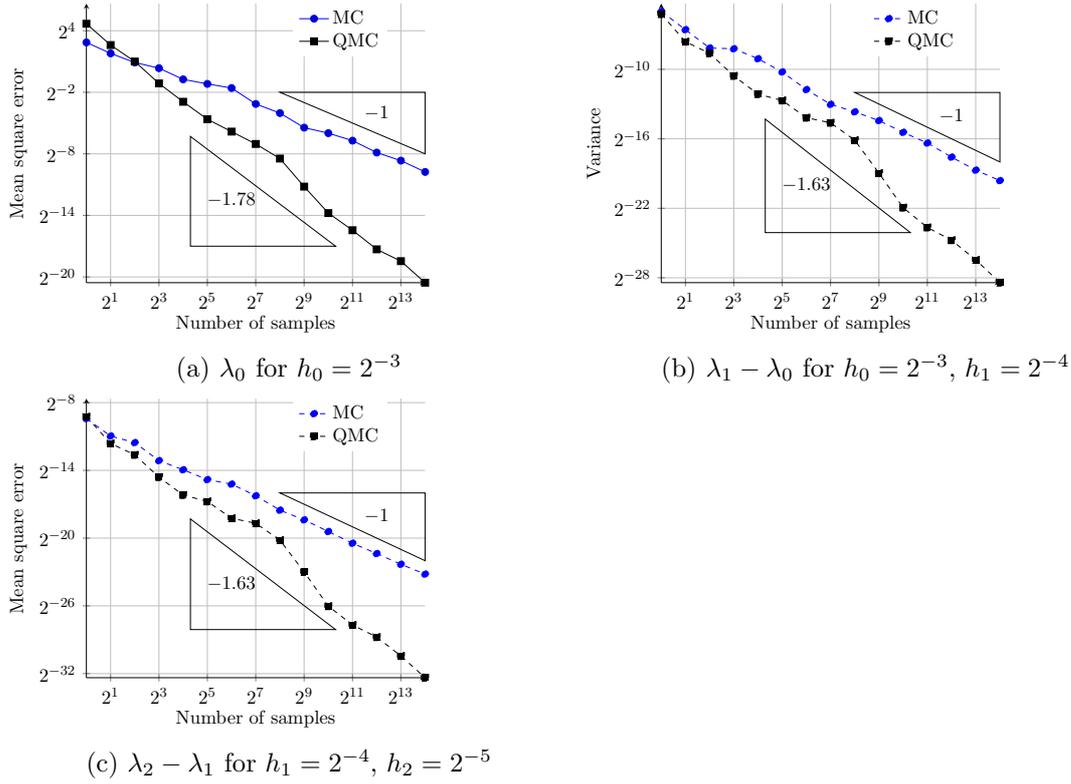
\begin{figure}[t!]

\begin{subfigure}{0.5\textwidth}
\begin{tikzpicture}[scale=0.65]
\begin{axis}[
   axis lines = left,
   grid = both,
   xlabel = {Number of samples},
	ylabel = {Mean square error},
   minor tick num=0,
   ylabel near ticks,
	ymode=log,
	xmode=log,
log basis y={2},
	log basis x={2},
   every axis y label/.style=
{at={(-0.2,0.5)},rotate=90},
   y tick label style={
       /pgf/number format/.cd,
           fixed,
           precision=100,
       /tikz/.cd,
   },
   legend style={at={(0.75, 1)}, anchor=north,legend columns=1,{draw=none}},
   legend cell align={left},
   ymax=1E2,
]
\addplot plot[mark=*,blue,domain=1:1e-30]
        file {results/qmc/mc_var.txt};
	\addlegendentry{MC}
\addplot plot[mark=square*,mark options={draw=black,fill=black},black,domain=1:1e-30]
        file {results/qmc/qmc_var.txt};
        \addlegendentry{QMC}

\draw (8,-2)
  -- (14,-2) 
  -- (14,-8)
  -- cycle;
\node at (12,-4) {$-1$};

\draw (10.3,-17)
  -- (4.3,-17) 
  -- (4.3,-6.3)
  -- cycle;
\node at (6,-12.5) {$-1.78$};

\end{axis}
\end{tikzpicture}
\caption{$\lambda_0$ for $h_0=2^{-3}$}
\end{subfigure}
\begin{subfigure}{0.5\textwidth}
\begin{tikzpicture}[scale=0.65]
\begin{axis}[
   axis lines = left,
   grid = both,
   xlabel = {Number of samples},
	ylabel = {Variance},
   minor tick num=0,
   ylabel near ticks,
	ymode=log,
	xmode=log,
log basis y={2},
	log basis x={2},
   every axis y label/.style=
{at={(-0.2,0.5)},rotate=90},
   y tick label style={
       /pgf/number format/.cd,
           fixed,
           precision=100,
       /tikz/.cd,
   },
   legend style={at={(0.75,1)}, anchor=north,legend columns=1,{draw=none}},
   legend cell align={left},
   ymax=5e-2,
]
 \addplot plot[mark=*,mark options={draw=blue,fill=blue},,blue,dashed,domain=1:1e-30]
        file {results/qmc/mc_var_l1.txt};
	\addlegendentry{MC}
\addplot plot[mark=square*,mark options={draw=black,fill=black},black, dashed,domain=1:1e-30]
        file {results/qmc/qmc_var_l1.txt};
        \addlegendentry{QMC}

\draw (8,-12)
  -- (14,-12) 
  -- (14,-18)
  -- cycle;
\node at (12,-14) {$-1$};

\draw (10.3,-24.1)
  -- (4.3,-24.1) 
  -- (4.3,-14.3)
  -- cycle;
\node at (6,-20) {$-1.63$};

\end{axis}
\end{tikzpicture}
\caption{$\lambda_1 - \lambda_0$ for $h_0 = 2^{-3}$, $h_1=2^{-4}$}
\end{subfigure}
\begin{subfigure}{0.5\textwidth}
\begin{tikzpicture}[scale=0.65]
\begin{axis}[
   axis lines = left,
   grid = both,
   xlabel = {Number of samples},
	ylabel = {Mean square error},
   minor tick num=0,
   ylabel near ticks,
	ymode=log,
	xmode=log,
log basis y={2},
	log basis x={2},
   every axis y label/.style=
{at={(-0.2,0.5)},rotate=90},
   y tick label style={
       /pgf/number format/.cd,
           fixed,
           precision=100,
       /tikz/.cd,
   },
   legend style={at={(0.75,1)}, anchor=north,legend columns=1,{draw=none}},
   legend cell align={left},
   ymax=5e-3,
]
\addplot plot[mark=*,mark options={draw=blue,fill=blue},blue,dashed,domain=1:1e-30]
        file {results/qmc/mc_var_l2.txt};
	\addlegendentry{MC}
\addplot plot[mark=square*,mark options={draw=black,fill=black},black, dashed,domain=1:1e-30]
        file {results/qmc/qmc_var_l2.txt};
        \addlegendentry{QMC}

\draw (8,-16)
  -- (14,-16) 
  -- (14,-22)
  -- cycle;
\node at (12,-18) {$-1$};

\draw (10.3,-28.1)
  -- (4.3,-28.1) 
  -- (4.3,-18.3)
  -- cycle;
\node at (6,-24) {$-1.63$};
\end{axis}
\end{tikzpicture}
\caption{$\lambda_2 - \lambda_1$ for $h_1 = 2^{-4}$, $h_2=2^{-5}$}
\end{subfigure}
\caption{Convergence of QMC and MC methods using tgRQI and SUPG for Test Case III with random velocity and conductivity. Plots (a), (b), (c) give the MSE of estimators versus sample sizes for grid sizes $h = 2^{-3}, 2^{-4}, 2^{-5}$, respectively. Blue lines with circles and black lines with squares indicate the MSE for MC and QMC, respectively. Dashed lines and solid lines correspond to the MSE of the estimated multilevel differences and the MSE of the estimated eigenvalues, respectively.}
\label{fig:qmc_error}
\end{figure}

\begin{figure}[t!]
\centering
\begin{tikzpicture}[scale=0.6]
\begin{axis}[
   axis lines = left,
   grid = both,
   xlabel = root mean square error,
	ylabel = {CPU time},
   minor tick num=0,
   ylabel near ticks,
	ymode=log,
	xmode=log,
log basis y={2},
	log basis x={2},
   every axis y label/.style=
{at={(-0.2,0.5)},rotate=90},
   y tick label style={
       /pgf/number format/.cd,
           sci,
           precision=100,
       /tikz/.cd,
   },
   legend style={at={(1.1,1.05)}, anchor=north,legend columns=1, draw=none},
   legend cell align={left},
]

\addplot plot[mark=triangle*,blue,domain=1:1e-30]
        file {results/eg3/mse_mc_new.txt};
	\addlegendentry{MC with Rayleigh quotient and SUPG}
\addplot plot[mark=triangle*,black,domain=1:1e-30]
        file {results/eg3/mse_new.txt};
	\addlegendentry{MLMC with Rayleigh quotient and SUPG}
 \addplot plot[mark=triangle*,red,domain=1:1e-30]
        file {results/qmc/mse_qmc_new.txt};
	\addlegendentry{MLQMC with Rayleigh quotient and SUPG}
\end{axis}
\end{tikzpicture}
\caption{CPU time vs.~root mean square error of the estimators in Test Case III.}
\label{fig:cpu_mse_v_random}
\end{figure}
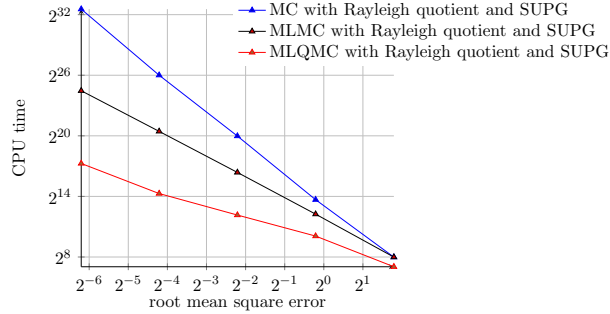

\section{Conclusion}
\label{sec:conc}
In this paper we have considered and developed various MLMC methods for stochastic convection-diffusion eigenvalue problems in~2D.
First, we established certain error bounds on the variational formulation of the eigenvalue problem under assumptions such as eigenvalue gap, boundedness, and other approximation properties.
Then we presented the MLMC method based on a hierarchy of geometrically refined meshes with and without homotopy. We also discussed how to improve the computational complexity of MLMC by replacing Monte Carlo samples with QMC samples. At last, we provided numerical results for three test cases with different convection velocities.

Test Case I shows that, for low convection velocity, all variants of the MLMC method (based on a Galerkin FEM discretization of the PDE) achieve optimal $O(\varepsilon^{-2})$ complexity, including the one with homotopy. In Test Case II with a high convection velocity, the homotopy-based MLMC does not work anymore --- at least without increasing the number of levels --- and MLMC based on Galerkin FEM has severe stability restrictions, preventing the use of a large number of levels. This restriction can be circumvented easily by using stable SUPG discretizations. Numerical experiments suggest that MLMC with SUPG achieves the optimal $O(\varepsilon^{-2})$ complexity and is 10-20 times faster than the Galerkin FEM-based versions for the same level of accuracy. In Test Case III, we considered both the conductivity and the convection velocity as random fields and compared the performance of MLMC and MLQMC. In this example, both MLMC and MLQMC deliver computational complexities that are close to the optimal complexities predicted by the theory, while the rate of the computational complexity of MLQMC outperforms that of MLMC. 


\section{Appendix: Bounding the constants in the FE error}
The results in Theorem~\ref{thm:fe} 
follow from the Babu\v{s}ka--Osborn theory \cite{BO91}. In this appendix we show that the constants can be bounded independently of the stochastic parameter.

The Babu\v{s}ka--Osborn theory studies  how the
continuous solution operators $T_{\bsomega}$, $T_{\bsomega}^* : V \to V$,
which for $f, g \in V$ are defined by
\begin{align*}
\calA(\bsomega; T_{\bsomega} f, v) \,&=\, \langle f, v \rangle
\quad\text{for all } v \in V,
\\
\calA(\bsomega; w, T_{\bsomega}^* g) \,&=\, \langle w, g \rangle
\quad \text{for all } w \in V,
\end{align*}
are approximated by the discrete operators 
$T_{\bsomega, h}, T_{\bsomega, h}^* : V_h \to V_h$,
\begin{align*}
\calA(\bsomega; T_{\bsomega, h} f, v_h) \,&=\, \langle f, v_h \rangle 
\quad\text{for all } v_h \in V_h,
\\
\calA(\bsomega; w_h, T_{\bsomega, h}^* g) \,&=\, \langle w_h, g \rangle
\quad \text{for all } w_h \in V_h.
\end{align*}

We summarize the pertinent details here.
First, we introduce:
\begin{align*}
\eta_h(\lambda(\bsomega)) \,&\coloneqq\, 
\sup_{u \in \calE(\lambda(\bsomega))} \inf_{\chi \in V_h} \|u - \chi\|_V,
\\
\eta_h^*(\lambda(\bsomega)) \,&\coloneqq\, 
\sup_{v \in \calE^*(\lambda(\bsomega))} \inf_{\chi \in V_h} \|v - \chi\|_V,
\end{align*}
where the eigenspaces are defined by
\begin{align*}
\calE(\lambda(\bsomega)) \,&\coloneqq\, \{ u : u \text{ is an eigenfunction of \eqref{eq:varevp} corresponding
to } \lambda(\bsomega), \|u\|_{L^2} = 1\},\\
\calE^*(\lambda(\bsomega)) \,&\coloneqq\, \{ u^* : u^* \text{ is an eigenfunction of \eqref{eq:dualevp} corresponding
to } \lambda(\bsomega), \|u^*\|_{L^2} = 1\}.
\end{align*}

The result for the eigenfunction \eqref{eq:u-fe-err} is given by \cite[Thm.~8.1]{BO91},
which gives
\begin{equation}
    \|u(\bsomega) - u_h(\bsomega)\|_V \,\leq\, 
    C_u(\bsomega) \eta_h(\lambda(\bsomega)),
\end{equation}
for a constant $C(\bsomega)$ defined below.
Since $\lambda(\bsomega)$ is simple, the best approximation property of $V_h$ in $H^2(D)$
followed by Theorem~\ref{thm:u-reg} gives
\begin{equation}
\label{eq:eta-bound}
\eta_h(\lambda(\bsomega))
    \,\leq\, C_{\mathrm{BAP}} \|u(\cdot, \bsomega)\|_{H^{2}} \,h
    \,\leq\, C_{\mathrm{BAP}} C_{2, \lambda} |\lambda(\bsomega)| \,h
    \,\leq\, C_{\mathrm{BAP}} C_{2, \lambda} \lambdahat \,h,
\end{equation}
where the best approximation constant $C_\mathrm{BAP}$ is independent of $\bsomega$.
In the last inequality we have also used that 
$\lambda(\bsomega)$ is continuous on the compact domain $\Omega$,
thus can be bounded uniformly by 
\begin{equation}
\label{eq:lambda-bound}
\lambdahat \,\coloneqq\, \max_{\bsomega \in \Omega} |\lambda(\bsomega)| \,<\, \infty.
\end{equation}

Hence, all that remains is to bound $C_u(\bsomega)$, uniformly in $\bsomega.$
This constant is given by
\begin{align*}
C_u(\bsomega) \,=\, 
&\|T_{\bsomega}\| \,\|u(\bsomega)\|_V 
\bigg( 1 + \frac{1}{\amin}\bigg) \frac{\length(\Gamma(\bsomega))}{\pi}
\\
&\cdot \sup_{\substack{z \in \Gamma\\h > 0}} 
\| R_z(T_{\bsomega, h}) \| 
\sup_{z \in \Gamma(\bsomega)} \|R_z(T_{\bsomega})\|,
\end{align*}
where $\Gamma(\bsomega)$ is a circle in the complex plane enclosing the eigenvalue 
$\mu(\bsomega) = 1/\lambda(\bsomega)$ of $T_{\bsomega}$, but no other points in the spectrum $\sigma(T_{\bsomega})$,
and for an operator $A$ and $z \in \rho(A) = \C \setminus \sigma(A)$, the \emph{resolvent set} of $A$, 
we define the resolvent operator $R_z(A) \coloneqq (z - A)^{-1}$.
Hence, all that remains is to show that $C_u(\bsomega)$ is bounded from above uniformly in $\bsomega$.

First, by the Lax--Milgram Lemma and the Poincar\'e inequality 
$T_{\bsomega}$ is bounded with $\|T_{\bsomega}\| \leq C_\mathrm{Poin}/ \amin$.
Also, since $\calA$ is coercive \eqref{eq:coerc} and $u(\bsomega)$ 
satisfies \eqref{eq:varevp}, using \eqref{eq:lambda-bound} we have the bound
\[
\|u(\bsomega)\|_V \, \leq \sqrt{\frac{\lambdahat}{\amin}}.
\]

Consider next the norm of the resolvent $\|R_z(T_{\bsomega})\|$ for $\bsomega \in \Omega$.
Note that care must be taken here since the domain for $z$, namely the resolvent set,
changes with $\bsomega$.

Let $\Gamma(\bsomega) = \{z \in \C : |z - \mu(\bsomega)| = \gamma/2\}$, where 
$\gamma$ is a lower bound on the spectral gap for $\mu$
\[
\gamma \,\coloneqq\, \inf_{\bsomega \in \Omega} \dist(\mu(\bsomega), \sigma(T_{\bsomega}) \setminus \{ \mu(\bsomega)\})
\,>\, 0.
\]
So that for each $\bsomega \in \Omega$ the circle $\Gamma(\bsomega)$ encloses only $\mu(\bsomega)$
and no other eigenvalues of $T_{\bsomega}$.
Then $z \in \Gamma(\bsomega)$ can be parametrised by both $\bsomega \in \Omega$ and $\theta \in [0, 2\pi]$,
\[
z \,=\, z(\bsomega, \theta) \,=\, \mu(\bsomega) + \frac{\gamma}{2}e^{i\theta} \,\in\, \Gamma(\bsomega).
\]
Clearly $z(\cdot, \cdot)$ is continuous in both $\bsomega$ and $\theta$ and belongs to the resolvent set, 
$z(\bsomega, \theta) \in \rho(T_{\bsomega})$, for all $\bsomega \in \Omega$ and $\theta \in [0, 2\pi]$.
Thus, $R_{z(\bsomega, \theta)}(T_{\bsomega})$ is bounded for all $\bsomega \in \Omega$ and $\theta \in [0, 2\pi]$.

For all $\bsomega \in \Omega$ we have the bound
\[
\sup_{z \in \Gamma(\bsomega)} \|R_z(T_{\bsomega})\| 
\,=\, \sup_{\theta \in [0, 2\pi]} \|R_{z(\bsomega, \theta)}(T_{\bsomega})\|
\,\leq\, \sup_{\substack{\theta \in [0, 2\pi]\\\bsomega \in \Omega}} \|R_{z(\bsomega, \theta)}(T_{\bsomega})\|.
\]

Now, in general the resolvent $R_z(A)$ is continuous in both arguments, $z$ and the (compact) operator $A$
(in fact it is holomorphic, see \cite[Theorem IV-3.11]{Kato84}).
Since $z$ is continuous in both $\theta$ and $\bsomega$ and $T_{\bsomega}$ is continuous in $\bsomega$,
it follows that  $R_{z(\bsomega, \theta)} (T_{\bsomega})$ is continuous in $\theta$ and $\bsomega$.
In turn, the norm $\|R_{z(\bsomega, \theta)}(T_{\bsomega})\|$
is also continuous in $\theta$ and $\bsomega$. Thus, $\|R_{z(\bsomega, \theta)}(T_{\bsomega})\|$
is bounded and continuous on the compact domain $[0, 2\pi] \times\Omega$, 
and so the maximum is attained for some 
$(\theta^*, \bsomega^*) \in [0, 2\pi] \times \Omega$, i.e., for all $\bsomega \in \Omega$
\[
\sup_{z \in \Gamma(\bsomega)}\|R_{z(\bsomega, \theta)}(T_{\bsomega})\| 
\,\leq\, \max_{\substack{\theta \in [0, 2\pi]\\\bsomega\in \Omega}} \|R_{z(\bsomega, \theta)}(T_{\bsomega})\|
\,=\, \|R_{z(\bsomega^*, \theta^*)}(T_{\bsomega^*})\| \,<\, \infty.
\]
For $h$ sufficiently small $\|R_z(T_{\bsomega, h})\|$ can be bounded in a similar way.

For $\Gamma(\bsomega)$ deined above $\length(\Gamma(\bsomega)) = \pi \gamma$, 
which is obviously independent of $\bsomega$. 
Thus, $C_u(\bsomega) \leq C_u < \infty$ for all $\bsomega \in \Omega$, where
\[
C_u\,\coloneqq\, \gamma \frac{C_\mathrm{Poin}}{\amin} \sqrt{\frac{\lambdahat}{\amin}}
\bigg( 1 + \frac{1}{\amin}\bigg)
\max_{\substack{\theta \in [0, 2\pi]\\\bsomega\in \Omega}} \|R_{z(\bsomega, \theta)}(T_{\bsomega})\|
\sup_{\substack{\theta \in [0, 2\pi]\\\bsomega\in \Omega\\h > 0}} \|R_{z(\bsomega, \theta)}(T_{\bsomega, h})\|
\]
is independent of $\bsomega$.

For the eigenvalue error \eqref{eq:lam-fe-err} we follow the proof of \cite[Theorem~8.2]{BO91}.
Since $\lambda(\bsomega)$ is simple, from Theorem~7.2 in \cite{BO91}, the eigenvalue
error is bounded by
\begin{align*}
|\lambda(\bsomega) - \lambda_h(\bsomega)|
&\leq\, 
C_\lambda(\bsomega) \eta_h(\lambda(\bsomega)) \eta_h^*(\lambda(\bsomega))
\,\leq\, C_\lambda(\bsomega) C_\eta h^2,
\end{align*}
where in the second inequality we have used \eqref{eq:eta-bound} and the equivalent 
bound for the dual eigenvalue, combining the two constants into $C_\eta$. 
By following \cite{BO91}, the constant $C_\lambda(\bsomega)$ can be bounded independently of $\bsomega$ in a similar way to $C_u(\bsomega)$.

\paragraph{Acknowledgements.}
T.~Cui and S.~Polishchuk acknowledge support from the Australian Research Council, under grant number CE140100049 (ACEMS). S.~Polishchuk acknowledges support from the School of Mathematics at Monash University. T.~Cui acknowledges travel support offered by the IWR at Heidelberg University. T.~Cui and R.~Scheichl further acknowledge support from the Erwin Schr\"odinger Institute for Mathematics and Physics at the University of Vienna. H.~De Sterck acknowledges support from NSERC of Canada.

\subsection*{Declarations}

\noindent\textbf{Conflict of interest.} The authors declare that they have no conflict of interest.

\noindent\textbf{Data Availability.} Not applicable.

\bibliographystyle{plain}

\end{document}